\documentclass[12pt]{amsart} 
\usepackage[T1]{fontenc}

\usepackage{mathtools, enumerate}
\usepackage{xparse, etoolbox}

\calclayout

\usepackage[abbrev, bibtex-style, nobysame]{amsrefs} 

\BibSpec{misc}{
 +{}{\PrintAuthors} {author}
 +{, }{ \textit} {title}
 +{}{ \parenthesize} {date}
 +{, }{ \url} {url}
 +{, }{ } {note}
 +{.}{ } {transition}
}

\BibSpec{book}{
  +{} {\PrintPrimary}        {transition}
  +{, } { \textit}           {title} 
  +{.} { }              {part}
  +{:} { \textit}           {subtitle}
  +{, } { \PrintEdition}        {edition}
  +{} { \PrintEditorsB}       {editor}
  +{, } { \PrintTranslatorsC}     {translator}
  +{, } { \PrintContributions}     {contribution}
  +{, } { }              {series}
  +{, } { \voltext}          {volume}
  +{, } { }              {publisher}
  +{, } { }              {organization}
  +{, } { }              {address}
  +{, } { \PrintDateB}         {date}
  +{, } { }              {status}
  +{} { \parenthesize}        {language}
  +{} { \PrintTranslation}      {translation}
  +{;} { \PrintReprint}        {reprint}
  +{.} {}               {transition}
}

\usepackage{tikz-cd} 
\usepackage{microtype} 
\usepackage{hyperref}

\newtheorem{Theorem}{Theorem}[section]

\newtheorem{Prop}[Theorem]{Proposition} 
\newtheorem{Corollary}[Theorem]{Corollary} 
\newtheorem{Lemma}[Theorem]{Lemma}

\theoremstyle{definition} 
\newtheorem*{Definition}{Definition} 

\theoremstyle{remark} 
\newtheorem*{Remark}{Remark} 
 
\newtheorem*{Example}{Example}

\renewcommand{\d}{\partial }

\newcommand{\rr}{\mathbb{R}}
\newcommand{\zz}{\mathbb{Z}}
\newcommand{\nn}{\mathbb{N}}

\newcommand{\cS}{\mathcal{S}}
\newcommand{\cC}{\mathcal{C}}
\newcommand{\cH}{\mathcal{H}}
\newcommand{\cD}{\mathcal{D}}

\DeclareMathOperator\diam{diam}
\DeclareMathOperator\myc{C}
\DeclareMathOperator\myh{H}

\DeclareMathOperator\Hom{Hom}
\DeclareMathOperator\Ext{Ext}
\DeclareMathOperator\Tor{Tor}
\newcommand{\symdiff}{\mathbin{\triangle}}

\newcommand{\C}[1][*]{\myc^{#1}}

\newcommand{\Cx}[1][*]{\myc{\mathrm X}^{#1}}
\newcommand{\cx}[1][*]{\myc{\mathrm X}_{#1}}

\newcommand{\cs}[1][*]{\myc_{#1}^{\mathrm s}}

\newcommand{\clf}[1][*]{\myc_{#1}^{lf}}
\newcommand{\cf}[1][*]{\myc_{#1}}

\renewcommand{\H}[1][*]{\myh^{#1}}

\newcommand{\Hx}[1][*]{\myh{\mathrm X}^{#1}}
\newcommand{\hx}[1][*]{\myh{\mathrm X}_{#1}}
\newcommand{\hlf}[1][*]{\myh_{#1}^{\mathrm lf}}
\newcommand{\h}[1][*]{\myh_{#1}}
\newcommand{\rH}[1][*]{\tilde{\myh}^{#1}}
\newcommand{\rh}[1][*]{\tilde{\myh}_{#1}}

\newcommand{\ga}{\alpha }

\newcommand{\gs}{\sigma }
\newcommand{\gt}{\tau}
\newcommand{\gr}{\rho}
\newcommand{\gD}{\Delta}
\newcommand{\gd}{\delta}
\newcommand{\gk}{\kappa}
\newcommand{\gl}{\lambda}
\newcommand{\go}{\omega}

\renewcommand{\ge}{\varepsilon}

\let\oldsubset\subset
\renewcommand{\subset}[1][]{\mathrel {\overset{\scriptscriptstyle #1}{\oldsubset}}}

\let\oldin\in
\renewcommand{\in}[1][]{\overset{#1}{\oldin}}

\let\oldnotin\notin
\renewcommand{\notin}[1][]{\overset{#1}{\oldnotin}}

\DeclarePairedDelimiter\abs{\lvert}{\rvert}
\newcommand{\Supp}[2][]{\abs{#2}_{#1}}
\newcommand{\SupP}[1]{\abs*{#1}}

\NewDocumentCommand\ccap{o}{\mathbin{\overset{\scriptscriptstyle \mathrm{c}
	\IfNoValueTF{#1}{} {, #1} }
	{\cap} } }

\NewDocumentCommand\cminus{o}{\mathbin{\oset{\mathrm{c}
	\IfNoValueTF{#1}{} {, #1} }
	{-} } }	
	
\NewDocumentCommand\ceq{o}{\mathbin{\overset{\scriptscriptstyle \mathrm{c}
	\IfNoValueTF{#1}{} {, #1} }
	{=} } }	
	\NewDocumentCommand\cneq{o}{\mathbin{\overset{\scriptscriptstyle \mathrm{c}
	\IfNoValueTF{#1}{} {, #1} }
	{\neq} } }	
	
\NewDocumentCommand\csubset{o}{\mathrel{\subset[\mathrm{c}\IfNoValueTF{#1}{} {, #1} ]}}
\NewDocumentCommand\he{o}{\mathbin{\overset{ \scriptscriptstyle{#1}}
	{\simeq} } }	
	
\makeatletter
\newcommand{\oset}[3][0ex]{
 \mathrel{\mathop{#3}\limits^{
  \vbox to#1{\kern 0\ex@
  \hbox{$\scriptscriptstyle#2$}\vss}}}}
\makeatother

\newcommand{\cupp}{\mathbin{\smile}}
\newcommand{\capp}{\mathbin{\frown}}

\title{Coarse cohomology of the complement}
\author{Arka Banerjee}
\address{Department of Mathematics and Statistics\\
Auburn University\\
Auburn, AL~36849}
\email{azb0263@auburn.edu}
\author{Boris Okun}
\address{Department of Mathematical Sciences\\
University of Wisconsin--Milwaukee\\
Milwaukee, WI~53201}
\email{okun@uwm.edu}

\begin{document}
\begin{abstract}
    In this paper we define the coarse (co)homology of the complement of a subspace in a metric space, generalizing the coarse (co)homology of Roe.
    We give a model space which encodes coarse geometric structure of the complement.
    We also introduce a new approach to coarse Poincaré duality spaces.
    We prove a version of coarse Alexander duality for these spaces and give a homological criterion for a space to be a coarse PD($n$) space.
    Our approach is inspired by the work of Kapovich and Kleiner, but is somewhat different, and we believe, simpler.
\end{abstract}

\maketitle

\section{Introduction}

The goal of this paper is to recast the coarse Alexander duality theorems of Kapovich and Kleiner \cite{kk05} into the framework of the coarse homology and cohomology theories of Roe~\cites{r93, r03}.
Recall that for $A$, a subcomplex of $S^{n}$, the classical Alexander duality is given by the isomorphisms $\rH(S^{n} -A ) \cong \rh[n-1-*](A)$ and $\rh(S^{n} -A ) \cong \rH[n-1-*](A)$.
The results of \cite{kk05} are similar in spirit in the context of subsets of coarse Poincaré duality spaces, but their actual statements are quite technical.
Our main result is the following theorem (Theorem \ref{t:pd}):
\begin{Theorem}
    If $X$ is a coarse PD($n$) space, then for any $A\subset X$
    \begin{align*}
        \Hx[k](X-A) &\cong \hx[n-k](A), \\
        \hx[k](X-A) &\cong \Hx[n-k](A).
    \end{align*}
\end{Theorem}
Our main innovation is in the definition of the left-hand sides of these formulas, and in fact in the introduction of a notion of the coarse complement.

One might think naively that, since in the coarse geometry one is allowed to replace any subset with its metric neighborhood, the coarse complement $X-A$ should be the complement of a (perhaps bigger and bigger) metric neighborhood, $X-N_r(A)$.
If $A$ is bounded, this is indeed the right picture, as all these complements are coarsely equivalent to each other (and in fact to $X$).
However, for unbounded subsets it is not a right notion.
For example, if $X$ is a Euclidean plane and $A$ is a line in $X$, then $X-N_r(A)$ are still all coarsely equivalent to $X$, contrary to the idea that a line coarsely separates a plane, and therefore the coarse topology should change.

Instead, we consider complements of \emph{expanding} neighborhoods $X-N_f(A)$, where $N_{f}(A)= \bigcup_{x\in A} B(x, f(x))$ for proper functions $f:A\to [0, \infty)$.
These complements form a directed system of metric spaces under inclusion (in the direction of slower growth of $f$), and we define coarse (co)chain complexes of the complement $X-A$ as the appropriate limits of the usual coarse complexes of these complements.

Our second observation, inspired by~\cite{msw11}, is that almost all relevant information is captured by the notion of \emph{coarse containment}.
We write $A\csubset B$ to mean that $A \subset N_{r}(B)$ for some $r$.
In some sense coarse geometry is the study of the poset of the equivalence classes of subsets of $X$ induced by this relation.
From this point of view the coarse complement $X\cminus A$ is the collection of classes of subsets which are coarsely disjoint from $A$, i.e.
have bounded intersection with any metric neighborhood of $A$.
This leads us to a very explicit description of the limiting complexes in terms of support, which is very similar to Roe's definition of the coarse cochains.
This description is key to our proof of Alexander duality.

Moreover, this description allows us to identify the coarse type of $X\cminus A$ with the coarse type of the metric space $(X/\Bar{A}, d_{A})$ whose metric comes from the pseudometric $d_A(x, y) = \min\{d(x, A)+d(y, A), d(x, y)\}$ on $X$.
This identification can be helpful in computations of $\Hx(X-A)$.

Our definition of a coarse PD($n$) space is a slight generalization of the definition of Kapovich and Kleiner, in terms of controlled chain homotopy equivalence between coarse chain and cochain complexes.
We do not require bounded geometry or simplicial structure. 
The classical Poincaré duality maps on manifolds are given by the cap product with the fundamental class in one direction and the slant product with the orientation class in the other.
We define a coarse version of these classes (called \emph{orientation pair}) and prove a similar statement for coarse PD($n$) spaces (Theorem~\ref{t:op}):
\begin{Theorem}
    $X$ is a coarse PD($n$) space if and only if it has an orientation pair.
\end{Theorem}
As a corollary, we obtain that uniformly acyclic topological manifolds are coarse PD($n$) spaces, a similar statement is claimed in~\cite{kk05} with only a sketch of the proof.
\begin{Remark}
    This theorem, as well as the Alexander duality, indicate that the term ``coarse PD($n$) space'' is perhaps a misnomer, a more appropriate name would be ``coarse homology manifold''.
\end{Remark}
\subsection{Overview} This paper is organized as follows.
In Section~\ref{s:prelim}, we introduce coarse language and state basic results about coarse containment.
We also describe several variations of Alexander--Spanier cochains and the corresponding cohomology theories.

Section~\ref{s:cohomology complements} contains our main definition of the coarse (co)homology of a complement.
The main result there is Lemma~\ref{X-A}, which gives an equivalent description of the cochain complex.
In Section~\ref{s:complements}, we discuss a more geometric version of our approach.
In Section~\ref{s:separation} we discuss the notion of coarse separation and show that elements in $\Hx[1](X-A;\zz_2)$ corresponds to different ways in which $A$ coarsely separates $X$.
In Section~\ref{s:PD} we introduce the notion of a coarse Poincaré duality space, and prove the main result of this paper, the coarse Alexander Duality, Theorem~\ref{t:pd}.
A standard argument then implies a coarse version of the Jordan separation theorem.
We study coarse versions of the standard (co)homological products in Section~\ref{s:products} and use them in Section~\ref{s:alg pd} to derive a homological criterion for a space to be a coarse PD($n$) space, and to prove that uniformly acyclic $n$-manifolds are coarse PD($n$) spaces.
Finally, in Section~\ref{s:Algebraic duality} we show that (under a mild set-theoretic assumption) integral coarse chains and coarse cochains are dual of each other in the sense of taking $\Hom( \cdot, \zz)$.
\subsection{Acknowledgements} We thank Kevin Schreve for many helpful conversations and valuable comments on the early draft of the paper.
The second author would like to thank the Max Planck Institute for Mathematics for its support and excellent working conditions.

\section{Preliminaries}\label{s:prelim}
\subsection{Coarse containment} We adopt the notation from~\cite{msw11}.
Let $(X, d)$ be a metric space.
For $A \subset X$ denote $N_{R}(A)=\{x \in X \mid d(x, A) <R \}$.
We will call such neighborhoods \emph{metric neighborhoods} of $A$.
We will say that $A$ is \emph{$R$-contained} in $B$, $A \subset[R] B$, if $A \subset N_{R}(B)$.
$A$ is \emph{coarsely contained} in $B$, $A \csubset B$, if $A \subset[R] B$ for some $R$.
Two subsets are \emph{coarsely equal}, $A \ceq B$, if $A \csubset B$ and $B \csubset A$.
All bounded subsets are obviously coarsely equal, and we declare the empty set to be coarsely equal to bounded subsets.
Formally, this requires special casing some of the metric arguments in the paper.
A way to avoid doing this is to add one extra point to $X$ and all its subsets and extend the metric.
We will simply ignore this annoyance.

We will use the following observation about the intersection of metric neighborhoods in later sections:
\begin{Lemma}\label{l:nbds}
    Let $A$, $B$ and $C$ be subsets of metric space $X$.
    If $C \subset[R] A$ and $C \subset[r] B$, then
    \[
    C \subset[R] A \cap N_{R+r}(B).
    \]
\end{Lemma}
\begin{proof}
    If $x \in C$, then there exist $y \in A$ such that $d(x, y) < R$.
    Therefore $y \in A \cap N_{R+r}(B)$, and hence $x \in N_R(A \cap N_{R+r}(B))$ by the triangle inequality.
\end{proof}

\subsection{Coarse intersection}
Recall from~\cite{msw11} the notion of \emph{coarse intersection}: $A \ccap B \ceq C$ if for all sufficiently large $R$, $N_{R}(A) \cap N_{R}(B) \ceq C$.
The coarse intersection is not always well-defined, it may happen that the coarse type of $N_{R}(A) \cap N_{R}(B) $ does not stabilize as $R$ goes to infinity.
However the notion ``coarse intersection is coarsely contained in'' is well-defined.
$A \ccap B \csubset C$ means that for any $R$, $N_{R}(A) \cap N_{R}(B) \csubset C$.
It is immediate from Lemma~\ref{l:nbds} that this condition is equivalent to the condition that for any $R$, $A \cap N_{R}(B) \csubset C$.

The following Lemma shows that coarse containment and intersection behave as expected.
\begin{Lemma}\label{l:subcap}
    Let $A$, $B$, $C$ and $D$ be subsets of metric space $X$.
    If $A \csubset B$, $A \csubset C$, and $B \ccap C \csubset D$, then $A \csubset D$.
\end{Lemma}
\begin{proof}
    By hypothesis, $A\subset N_R(B)$ and $A\subset N_R(C)$ for some $R$.
    Also we have $N_R(B)\cap N_R(C)\subset N_{R'}(D)$ for some $R'$.
    It follows that $A\subset N_{R'}(D)$.
\end{proof}
A subset $C$ is \emph{coarsely disjoint} from $A$ if $A \ccap C \ceq *$, in other words if $A \cap N_{R}(C)$ is bounded for all $R$.
Note that coarse disjointness is independent of the ambient space, it depends only on the metric on $A\cup C$.

\subsection{Coarse difference} We will write $C \csubset B \cminus A $ to mean that $C \csubset B$ and $A \ccap C \ceq *$.
One should think of the coarse difference $B \cminus A $ as a collection of subsets, it rarely has a well-defined coarse type in $X$.
Nevertheless, the notion of coarse containment between differences is well-defined.
We will write $( D \cminus C ) \csubset ( B \cminus A )$ to mean that $Y \csubset D \cminus C $ implies $Y \csubset B \cminus A $.
Unwinding the definitions, this really means that any subset of any metric neighborhood of $D$ which has bounded intersection with $C$ is contained in some metric neighborhood of $B$ and has bounded intersection with $A$.
\begin{Lemma}\label{l:complement}

    $B \ccap D \csubset A$ if and only if $(D \cminus A ) \csubset ( D \cminus B) $.
\end{Lemma}
\begin{proof}
    The forward direction is obvious.
    To prove the converse, suppose $B \ccap D \not\csubset A$, then the intersection $D$ and some metric neighborhood of $B$ contains a sequence of points going away from $A$.
    This sequence is in $D \cminus A $ but not in $D \cminus B$.
\end{proof}
\subsection{Coarse maps} 
Suppose $f:X\to Y$ is a map between metric spaces.
$f$ is \emph{proper} if $f^{-1}(B)$ is bounded in $X$ whenever $B$ is bounded in $Y$.
$f$ is \emph{coarse} if it is proper and has an upper control: there exists a function $\rho:[0, \infty)\to [0, \infty)$, so that
\[
d(f(x), f(y))\leq \rho(d(x, y)) \quad \text{for all } x, y\in X.
\]
$f$ is a \emph{coarse embedding} if there exists two proper functions $\rho_-,\rho_+:[0,\infty) \to [0,\infty)$ such that
\[
\rho_-(d(x,y))\leq d(f(x),f(y))\leq \rho_+(d(x,y)).
\]
$f$ is a \emph{coarse equivalence} if it is a coarse embedding and $Y\csubset f(X)$.

Two maps $f, g : X\to Y$ are \emph{close} if $d(f(x), g(x))$ is uniformly bounded.
It is easy to see that a coarse map $f:X \to Y$ is a coarse equivalence if and only if it has a coarse inverse: i.e.
if there exists a coarse map $g:Y \to X$, such that the compositions $fg$ and $gf$ are close to the identities.

The following lemma is immediate from the definitions.
\begin{Lemma}\label{l:uc}
    Maps with an upper control preserve the coarse containment, and the preimages of coarsely disjoint sets under a coarse map are coarsely disjoint.
\end{Lemma}

A map of pairs $f:(X, A)\to (Y, B)$ is \emph{relatively proper} if for any $D \csubset B$, $f^{-1}(D) \csubset A$, and it is \emph{relatively coarse} if it is relatively proper and has an upper control.

Relatively coarse maps induce coarse maps on the coarse complements, in the following sense.
\begin{Lemma}\label{l:rel coarse map}
    Let $f:(X, A) \to (Y, B)$ be a relatively coarse map and $C\csubset X\cminus A$.
    Then $f(C)\csubset Y\cminus B$ and $f|_C$ is a coarse map.
\end{Lemma}
\begin{proof}
    To prove the first claim, we need to show that for any $D \csubset B$, $f(C)\cap D$ is bounded.
    Since $f$ is relatively proper, $f^{-1}(D) \csubset A$.
    Therefore,
    \[
    C\cap f^{-1}( D) \csubset C\ccap A \ceq *,
    \]
    and thus $C\cap f^{-1}( D) $ is bounded.
    Since $f(C)\cap D = f( C\cap f^{-1}( D) )$ and $f$ has an upper control $f(C)\cap D$ is bounded.

    To prove $f|_C$ is a coarse map, we only need to show $f|_C$ is a proper map because we already have an upper control by definition.
    Let $D\subset Y$ be bounded.
    Then $D\csubset B$ and, since $f$ is relatively proper, $f^{-1}(D) \csubset A$.
    Thus
    \[
    f|_{C}^{-1}(D)=f^{-1}(D) \cap C \csubset A \ccap C \ceq *.
    \]
    Hence $f|_{C}^{-1}(D)$ is bounded, and $f|_C$ is proper.
\end{proof}
\subsection{Alexander--Spanier complexes} 
We will refer to points in $X^{n+1}$ as \emph{$n$-simplices}.
Let $G$ be an abelian group (we will be mostly concerned with $G=\zz$ or $\zz/2$). 
We will think of functions $X^{n+1} \to G$ as $n$-cochains or $n$-chains on $X$, depending on the context, and in the latter case will use additive notation $c=\sum_{\gs\in X^{n+1}}c_\gs \gs$.

The basic complex is the complex of finitely supported integral chains
\[
\cf(X):=\{c=\sum_{i=1}^k c_i \sigma_i \mid c_i \in \zz, \, \sigma_i\in X^{*+1}\}
\]
equipped with the usual boundary map, defined on the basis by
\[
\partial (x_0, \dots, x_n):=\sum_{i=0}^n (-1)^i(x_0, \dots, \hat{x}_i, \dots, x_n).
\]
Tensoring it with $G$ gives the complex of finitely supported chains with coefficients in $G$.
\[
\cf(X;G):=\{c=\sum_{i=1}^k c_i \sigma_i \mid c_i \in G, \, \sigma_i\in X^{*+1}\}
\]
Taking $\Hom(\cf(X), G)$ gives the complex of all Alexander--Spanier cochains
\[
\C[*](X)=\{\phi: X^{*+1} \to G\}
\]
with the coboundary operator
\[
d(\phi)(x_0, \dots, x_n)=\sum_{i=0}^n(-1)^{i} \phi(x_0, \dots, \hat{x}_i, \dots, x_n).
\]
These are acyclic complexes.

Note that the boundary map $\partial$ is well-defined on a larger complex of \emph{locally finite} chains $\clf(X; G)$ which consists of chains $c$ satisfying the following property: for any bounded $B \subset X$ only finitely many simplices in $c$ have vertices in $B$.

\subsection{Coarse (co)homology} 
In what follows, we will need to measure distances between simplices of different dimensions.
A convenient way to do this is to stabilize simplices by repeating the last coordinate, as follows.
Denote by $X^\infty$ the subset of the product of countably many copies of $X$, consisting of eventually constant sequences.
Equip $X^{\infty}$ with the $\sup$ metric.
Let $i:X^{n+1} \to X^\infty$ denote the map $ (x_{0}, \dots, x_{n}) \mapsto (x_{0}, \dots, x_{n}, x_{n}, x_{n}, \dots)$.
For a function $\phi:X^{n+1} \to \zz$ define its stabilized support
\[
\Supp{\phi} = \{i(\gs) \mid \gs \in X^{n+1} \text{ and } \phi(\gs) \neq 0\} \subset X^{\infty}.
\]
Let $\gD=i(X)$ denote the diagonal of $X^{\infty}$.
Define the support of $\phi$ at scale $r$ to be $\Supp[r]{\phi} = \Supp{\phi} \cap N_{r}(\gD)$.

We now define coarse (co)homology theories, following Roe~\cite{r93} and Hair~\cite{h10} using our language.
The coarse cochain complex is
\[
\Cx(X; G) := \{\phi \in \C(X; G) \mid \Supp[r]{\phi} \text{ is bounded for all }r \}.
\]
An equivalent way to define coarse cochains is to require the support to be coarsely disjoint from the diagonal:
\[
\Cx(X; G) = \{\phi \in \C(X; G) \mid \Supp{\phi} \ccap \gD \ceq * \}.
\]
The coboundary operator $d$ preserves this property, and the coarse cohomology $\Hx(X; G)$ is defined to be the cohomology of the complex $(\Cx(X; G),d)$.
The coarse homology $\hx(X; G)$ is the homology of the following subcomplex of $\clf(X; G)$
\[
\cx(X; G):=\{c\in \clf(X; G)\mid \Supp{c}\csubset \gD\}
\]
equipped with the restriction of the boundary operator $\partial$.
Note that in the presence of the support condition $\Supp{c}\csubset \gD$ local finiteness is equivalent to $\Supp{c}$ having finite intersections with bounded subsets of $X^{\infty}$.
\begin{Example}
    Roe~\cite{r93} showed that for uniformly contractible proper metric spaces the coarse cohomology is isomorphic to the compactly supported Alexander--Spanier cohomology.
    In particular, this applies to the universal cover of finite aspherical complexes.
    In this case the coarse homology is isomorphic to the locally finite homology~\cite{r96}*{Chapter 2}.
    For example,
    \[
    \hx(\rr^n; G)=\Hx(\rr^n; G)=
    \begin{cases}
        G & *=n, \\
        0 & \text{otherwise}.
    \end{cases}
    \]
\end{Example}

\subsection{Change of coefficients} 
A short exact sequence of abelian groups $0 \to G' \to G \to G'' \to 0$ induces short exact sequences of coarse complexes, and hence long exact sequences
\[
\dots \to \hx(X; G') \to \hx(X; G) \to \hx(X; G'') \to \hx[*-1](X; G') \to \dots
\]
and
\[
\dots \to \Hx(X; G') \to \Hx(X; G) \to \Hx(X; G'') \to \Hx[*+1](X; G') \to \dots
\]
It follows that taking coarse complexes commutes with finite direct sums.
For a finitely generated abelian group $G$ this leads to an identification of the coarse complexes with coefficients in $G$ with the tensor product with $G$ of the integral coarse complexes, as follows.
Pick a finitely generated free abelian resolution of $G$:
\[
0 \to F_{1} \to F_{0} \to G \to 0.
\]
Tensoring it with $\cx(X; \zz)$ and comparing to the induced short exact sequence of the coarse chain complexes gives a commutative diagram:
\[
\begin{tikzcd}[column sep=small] 
    0 \ar[r] & \cx(X; \zz) \otimes F_{1} \ar[r] & \cx(X; \zz) \otimes F_{0} \ar[r] & \cx(X; \zz) \otimes G \ar[r] & 0 \\
    0 \ar[r] & \cx(X; F_{1}) \ar[r] \ar[u, equal] &\cx(X; F_{0}) \ar[r] \ar[u, equal] & \cx(X; G) \ar[r] \ar[u, equal, dashed ]& 0
\end{tikzcd}
\]
and similarly for cochains.
\begin{Remark}
    We are not aware of general universal coefficient formulas in the coarse setting.
    The difficulty here is that neither chains nor cochains are free abelian groups.
    However, for proper uniformly contractible spaces, since the coarse theories are isomorphic to the compactly supported cohomology and the locally finite homology~\cite{r93}, we have~\cite{m78b}*{Cor.
    4.18 and 4.25}:
    \begin{align*}
        0 \to \Ext(\Hx[*+1](X; \zz), G) \to &\hx(X; G) \to \Hom(\Hx(X; \zz), G) \to 0, \\
        0 \to \Hx(X; \zz)\otimes G \to &\Hx(X; G) \to \Tor(\Hx[*+1](X; \zz), G) \to 0.
    \end{align*}
    One might expect these formulas to hold in greater generality.
\end{Remark}

\subsection*{Functoriality} 
Coarse maps induce maps on the coarse (co)homology.
The following proposition goes back to Roe, cf.~\cite{r93}*{2.13}.
\begin{Prop}\label{coarse map}
    A coarse map $f:X\to Y$ induces chain maps $f_*:\cx(X)\to \cx(Y)$ and $f^*:\Cx(Y)\to \Cx(X)$ by the usual formulas.
    \[
    f_*((x_0, \dots, x_n)):=(f(x_0), \dots, f(x_n)),
    \]
    \[
    f^*\phi:=\phi\circ f_*.
    \]
    Close coarse maps induce chain homotopic maps.
\end{Prop}
\begin{proof}
    Note that $f$ induces a coarse map $f:X^{\infty} \to Y^{\infty}$, and that proper maps preserve local finiteness, so the claims follow from Lemma~\ref{l:uc}.

    For the second claim, take two close coarse maps $f,g:X\rightarrow Y$.
    Then one can easily check that the map 
    \[
    D_n(x_0,x_1,\ldots, x_n):=\sum_{i=0}^n(f(x_0),f(x_1),\dots, f(x_i),g(x_i),\dots, g(x_n))
    \]
   extends to a map $\cx[n](X) \to \cx[n+1](Y)$ and gives a chain homotopy between $f_*$ and $g_*$.
   Similarly, the dual $D_n^*$ maps $\Cx[n+1](Y)$ to $\Cx[n](X)$ and gives a chain homotopy between $f^*$ and $g^*$.
   \end{proof}
In particular, coarsely equivalent spaces have isomorphic coarse (co)homology.

\section{Coarse (co)homology of complements}\label{s:cohomology complements}
We now introduce coarse complements into the picture.
Since the results in this section apply to any choice of coefficients and to simplify notation, we fix a coefficient group $G$ and omit it from the notation.

Let $A\subset X$.
A straightforward approach to define the coarse cohomology of the complement $X-A$ is to look at cochains on $X$ which restrict to coarse cochains on complements of expanding neighborhoods of $A$, as follows.
Given a proper function $f:A \to \rr_{\geq 0}$, the corresponding expanding neighborhood of $A$ is obtained by taking the union of open balls of radius $f(x)$ centered at $x$ as $x$ varies throughout $A$:
\[
N_{f}(A)= \bigcup_{x\in A} B(x, f(x)).
\]
Note that $f$ can take $0$ as its value and so $N_{f}(A)$ might not be an actual neighborhood of $A$.

The complements of expanding neighborhoods $X-N_{f}(A)$ form a directed system under the inclusion, and we are interested in the limit in the direction of slower growth of $f$.
In other words, we want the complement to be as large as possible.

Then the coarse chains and cochains on the complements of expanding neighborhoods $X-N_{f}(A)$ form a direct and an inverse system, respectively.
We define the coarse chains and cochains of the complement to be the (co)limits of these systems.
Note that the system of complements $X-N_{f}(A)$ is cofinal in the system of sets coarsely disjoint from $A$, since for any such set $C\csubset X\cminus A$, the map $f:A\to \rr_{\geq 0}$ given by $x\mapsto d(x, C)/2$ is proper and $C\subset X-N_{f}(A)$.
Thus, we have
\begin{align*}
    \cx[n](X-A) &= \{ c =\sum_{\gs\in X^{n+1}}a_\gs\gs \mid \exists C \csubset X \cminus A \quad c \in \cx[n](C) \}, \\
    \Cx[n](X-A)&= \{ \phi \in \C[n](X) \mid \forall C \csubset X \cminus A \quad \phi|_{C}\in \Cx[n](C)\}.
\end{align*}

The usual (co)boundary operators make $\cx(X-A)$ and $\Cx(X-A)$ into complexes, and we call their (co)homology \emph{the coarse (co)homology of the complement} $\hx(X-A)$ and $\Hx(X-A)$.
Note that $X-A$ here is a part of notation, not a set-theoretic complement.
\begin{Remark}
    If $A$ is a bounded subset of $X$, then every subset of $X$ is coarsely disjoint from $A$.
    It follows that, $\Cx(X-A)=\Cx(X)$ and $\cx(X-A)=\cx(X)$.
\end{Remark}

\subsection{Relatively coarse maps} 
We now state an analog of Proposition~\ref{coarse map} in the coarse complement setting.
\begin{Prop}
    A relatively coarse map $f:(X, A)\to (Y, B)$ between pairs induces chain maps $f_*:\cx(X-A)\to \cx(Y-B)$ and $f^*:\Cx(Y-B)\to \Cx(X-A)$.
\end{Prop}
\begin{proof}
    Suppose $c\in \cx(X-A)$.
    We want to show that $f_*(c)\in \cx(Y-B)$.
    Since $c\in \cx(X-A)$, there exists a $C\csubset X\cminus A$ such that $c\in \cx(C)$.
    By Lemma~\ref{l:rel coarse map}, $f|_C$ is a coarse map and hence Proposition~\ref{coarse map} implies that $f_*(c)\in \cx(f(C))$.
    Moreover by Lemma~\ref{l:rel coarse map}, $f(C)\csubset Y\cminus B$.
    It follows that $f_*(c)\in \cx(Y-B)$.

    Next we prove that $f^*(\phi)\in \Cx(X-A)$ if $\phi\in \Cx(Y-B)$.
    Let $C\csubset X\cminus A$.
    It is enough to show that $f^*\phi|_C\in \Cx(C)$.
    By Lemma~\ref{l:rel coarse map}, $f(C)\csubset Y\cminus B$.
    Since $\phi\in \Cx(Y-B)$, we have $\phi|_{f(C)}\in \Cx(f(C))$.
    Moreover, $f|_C$ is a coarse map by Lemma~\ref{l:rel coarse map}.
    Hence we have $f^*(\phi)|_C=({f|_C})^*(\phi|_C)\in \Cx(C)$ by Proposition~\ref{coarse map}.
    and the claim follows.
\end{proof}

\subsection{Another description} Recall that $i:X\to X^\infty$ denote the map $x\mapsto (x, \dots, x, \dots)$ and $\gD=i(X)$.
For a subset $A\subset X$, we denote the set $i(A)\subset \gD$ by $\gD_A$.
The next two lemmas provide more explicit descriptions of $\cx[n](X-A)$ and $\Cx[n](X-A)$.
\begin{Lemma}\label{X-A}
    \begin{align*}
        \cx[n](X-A)&=\{ c \in \clf[n](X) \mid \Supp{c} \csubset \gD \cminus \gD_{A} \}, \\
        \Cx[n](X-A)&= \{ \phi \in \C[n](X) \mid \Supp{\phi} \ccap \gD \csubset \gD_A \}.
    \end{align*}
\end{Lemma}
\begin{proof}
    The condition $c \in \cx[n](C)$ for some $C \csubset X \cminus A$ implies that $\Supp{c} \csubset \gD_{C}$.
    Therefore $\Supp{c} \csubset \gD \cminus \gD_{A}$.
    Vice versa, if $\Supp{c} \csubset \gD \cminus \gD_{A}$, set $C=\bigcup_{i=0}^{n} p_{i}(\Supp{c})$.
    Then $C \csubset X \cminus A$ and $\Supp{c} \subset C^{n+1}$, so $c \in \cx[n](C)$.
    This proves the first equation.

    To prove the second, note that the condition $\phi |_{C}$ is coarse for all $C \csubset X \cminus A$ means that $\Supp{\phi} \cap C^{n+1}$ is coarsely disjoint from $\gD_{C}$. 
    Since we can replace $C$ with its metric neighborhood, this is equivalent to $\Supp{\phi} \ccap C^{n+1} \ccap \gD_{C} \ceq *$, which we can rewrite as $\Supp{\phi} \ccap \gD_{C} \ceq * $.
    The condition $C \csubset X \cminus A$ is equivalent to $\gD_{C} \csubset \gD \cminus \gD_{A}$, so we have $( \gD \cminus \gD_{A} ) \csubset ( \gD \cminus \Supp{\phi} ) $.
    By Lemma~\ref{l:complement} this is equivalent to $\Supp{\phi} \ccap \gD \csubset \gD_{A}$, as claimed.
\end{proof}

We now consider a slight generalization of the above (co)chain complexes.
Let $A, B \subset X$.
Define:
\begin{align*}
    \cx(B-A\subset X)&=\{ c \in \clf(X) \mid \Supp{c} \csubset \gD_{B} \cminus \gD_{A} \}, \\
    \Cx(B-A \subset X)&= \{ \phi \in \C(X) \mid \Supp{\phi} \ccap \gD_B \csubset \gD_A \}.
\end{align*}
We claim that the homology of these complexes does not depend on the ambient space $X$, it depends only on the metric on $B \cup A$:
\begin{Lemma}\label{l:equiv1}
    The maps induced by inclusion
    \begin{align*}
        i_{*} :\cx(B-A \subset B \cup A) \to \cx(B-A \subset X), \\
        i^{*}:\Cx(B-A \subset X) \to \Cx(B-A \subset B \cup A)
    \end{align*}
    are chain homotopy equivalences.
\end{Lemma}
\begin{proof}
    Choose an approximate projection $\pi :X \to B \cup A$ with the property that $\pi$ restricts to the identity on $B \cup A$ and for all $x \in X$, $d(x, \pi(x)) < d(x, B \cup A) +1$.
    This map induces a map on $n$-simplices $\pi :X^{n+1} \to (B \cup A)^{n+1}$, and we claim that it extends linearly to a well-defined map $\pi_{n}: \cx[n](B-A \subset X) \to \cx[n](B-A \subset B \cup A)$.

    Indeed, if $c \in \cx[n](B-A \subset X)$, then there exists $R$ such that $\Supp{c} \subset[R] \gD_B$.
    Since $\pi |_{N_{R} (\gD_{B \cup A})}$ is $(R+1)$-close to the identity, it follows that $\pi_{n}(c)$ is a well-defined chain in $\cx[n](B-A \subset B \cup A)$.

    Next define
    \[
    D_n(x_0, \dots, x_n)=\sum_{i=0}^{n} (-1)^i(x_0, \dots, x_i, \pi(x_i), \dots, \pi(x_n)).
    \]
    A similar argument shows that $D_{n}$ extends linearly to a well-defined map $D_{n}: \cx[n](B-A \subset X) \to \cx[n+1](B-A \subset X)$.
    Finally, we note that $\pi_{*} i_{*}=id$ and $i_{*} \pi_{*} -id =D_{*}\partial + \partial D_{*} $.

    The approximate projection $\pi$ induces a map $\pi^*:\C(B \cup A) \to \C(X)$, $\pi^{*}(\phi)= \phi\pi$.
    Thus we have $\Supp{\pi^{*}(\phi)} \cap N_{R}( \gD_{A} ) \subset (\pi |_{N_{R}( \gD_{A} )})^{-1}(\Supp{\phi})$ and $\Supp{\pi^{*}(\phi)} \cap N_{R}( \gD_{B} ) \subset (\pi |_{N_{R}( \gD_{B} )})^{-1}(\Supp{\phi})$.
    Since $\pi |_{N_{R}( \gD_{B\cup A} )}$ is close to identity, it follows that $\pi^{*}$ restricts to a map $\pi^*:\Cx(B-A \subset B \cup A) \to \Cx(B-A \subset X)$.
    Similarly $D^{*}$ defined by $D^{*}(\phi) = \phi D_{*}$ gives a chain homotopy between $\pi^*i^{*} $ and $id$.
\end{proof}

It follows that we can write unambiguously $\hx(B - A)$ and $\Hx(B - A)$.
Note that when $B=X$ or $A=\emptyset$ this agrees with the previous notation.
Also, these definitions make sense even if $B$ is empty, and in this case the complexes become the Alexander--Spanier complexes of finitely supported chains and all cochains on $X$ and hence
\[
\hx(\emptyset; G)=\Hx(\emptyset; G)=
\begin{cases}
    G & *=0, \\
    0 & \text{otherwise}.
\end{cases}
\]

\section{Coarse complement as a space}\label{s:complements}

As we already mentioned, $X \cminus A$ rarely has a coarse type of a subspace of $X$.
However, it is still possible to compute the coarse cohomology of the complement as the coarse cohomology of a single space, but this space is not a subspace of $X$, but rather a quotient space.
We continue to suppress the coefficients $G$ from the notation.

We start by defining a pseudometric $d_A$ on $X$ by
\[
d_A(x, y) = \min(d(x, A)+d(y, A), d(x, y)).
\]
\begin{Lemma}\label{l:d-da}
    \leavevmode
    \begin{enumerate}
        \item\label{i:nbdA}
        $N_r^{d}(A^{n+1}) = N_r^{d_A}(A^{n+1})$.
        \item\label{i:d}
        $N_r^{d}(\gD) \subset N_r^{d_A}(\gD)$.
        \item\label{i:da}
        $N_r^{d_A}(\gD) \subset N_{2r}^{d}(\gD \cup A^{n+1})$.
    \end{enumerate}
\end{Lemma}
\begin{proof}

    \eqref{i:nbdA} and \eqref{i:d} follow from the facts that $d_{A}(x, A)=d(x, A)$ and $d_A(x, y) \leq d(x, y)$.

    To prove \eqref{i:da}, let $\gs \in N_{r}^{d_{A}}(\gD)$.
    So we have $d_A(\gs, x ) < r$, for some $x \in \gD$.
    If $d(x, A^{n+1}) \leq r$, then $d(\gs, A^{n+1})=d_{A}(\gs, A^{n+1}) \leq d_{A}(\gs, x)+ d_{A}(x, A^{n+1}) < 2r$ and $\gs \in N_{2r}^{d}(A^{n+1})$.
    If $d(x, A^{n+1}) > r$ then, since $d_A(x, \gs) < r$, it follows from the definition of $d_{A}$ that $d_{A}(x, \gs) = d(x, \gs)$ and $\gs \in N_{2r}^{d}(\gD)$.
\end{proof}
Much of the coarse metric theory works the same for pseudometric spaces.
In particular, we have the usual definition of coarse (co)chains.
\begin{align*}
    \cx[n](X, d_A) &= \{c \in \clf[n](X, d_{A}) \mid \Supp{c} \csubset[d_A] \gD \}, \\
    \Cx[n](X, d_A) &= \{\phi\in \C[n](X) \mid \Supp{\phi} \ccap[d_A] \gD \ceq[d_A ] * \}.
\end{align*}
\begin{Prop}\label{(X, d_A)=X-A}
    \begin{align*}
        \cx[n](X, d_A) &= \cx[n](X-A), \\
        \Cx[n](X, d_A) &= \Cx[n](X-A).
    \end{align*}
\end{Prop}
\begin{proof}

    Let $c \in \cx[n](X, d_A)$.
    Since $c \in \clf[n](X, d_A)$ and any bounded set in $(X^{n+1}, d)$ is bounded in $(X^{n+1}, d_A)$, we have $c \in \clf[n](X)$.
    Moreover, since $A^{n+1}$ is bounded in $d_{A}$ metric, Lemma~\ref{l:d-da}\eqref{i:nbdA} implies that $\Supp{c}\cap N^d_r(A^{n+1})$ is finite for any $r\geq 0$.
    So $\Supp{c}$ is coarsely disjoint from $A^{n+1}$.

    On the other hand, the condition $\Supp{c}\csubset[d_A]\gD$ and Lemma \ref{l:d-da}\eqref{i:da} give $\Supp{c} \csubset[d] \gD \cup A^{n+1}$.
    Hence, $\Supp{c}\csubset[d] \gD\cminus[d]\gD_A$, which together with local finiteness implies $c\in \cx[n](X-A)$.

    Conversely, let $c \in \cx[n](X-A)$.
    Then, for any $r$, $\Supp{c} \cap N^d_r(\gD_A) $ is bounded, and hence finite, since $c$ is locally finite.

    Then, since $\Supp{c} \csubset[d] \gD$, $\Supp{c} \cap N^d_r(A^{n+1}) $ is also finite, and therefore, by Lemma \ref{l:d-da}\eqref{i:nbdA}, we have that $ \Supp{c} \cap N_r^{A^{n+1}}(\gD_A) $ is finite as well.
    Also, the condition $\Supp{c} \csubset[d] \gD$ implies by Lemma \ref{l:d-da}\eqref{i:d} that $\Supp{c} \csubset[d_A] \gD$.
    Hence $c\in \cx[n](X, d_A)$.
    This proves the first assertion.

    Let $\phi \in \Cx[n](X, d_A)$.
    Then $\Supp{\phi} \ccap[d_{A}] \gD \ceq[d_{A}] *$.
    Since $A^{n+1}$ is bounded in $(X^{n+1}, d_{A})$, this is equivalent to $\Supp{\phi} \ccap[d_{A}] \gD \csubset[d_{A}] A^{n+1}$.
    By Lemma~\ref{l:d-da}\eqref{i:nbdA} and \eqref{i:d}, we have $\Supp{\phi} \ccap[d] \gD \csubset[d] A^{n+1}$.
    It follows that $( \Supp{\phi} \ccap[d] \gD ) \csubset[d] (A^{n+1} \ccap[d] \gD) = \gD_{A}$.
    Hence, $\phi \in \Cx[n](X-A)$.

    Conversely, if $\phi\in \Cx[n](X-A)$, then $\Supp{\phi} \ccap[d] \gD \csubset[d] \gD_A$, and it follows that $\Supp{\phi} \ccap[d] (\gD \cup A^{n+1}) \csubset[d] A^{n+1}$.
    By Lemma~\ref{l:d-da}\eqref{i:nbdA} and \eqref{i:da} this implies $\Supp{\phi} \ccap[d_{A}] \gD \csubset[d_{A}] A^{n+1}$.
    Hence $\phi\in \Cx[n](X, d_A)$.
    This proves the second assertion.
\end{proof}
Since $d_A(x, y)=0$ if and only if $(x, y)\in \overline{A} \times \overline{A}$, the pseudometric $d_A$ becomes a metric on the quotient space $X/ \overline{A}$.
One easily checks that the quotient map $q:(X, d_A) \to (X/\overline{A}, d_A)$ is a coarse equivalence, and therefore it induces an isomorphism on the coarse (co)homology.
Using the previous proposition we obtain the following.
\begin{Prop}\label{(X-A)=(X/A)}
    The quotient map $q: X \to X/\overline{A}$ induces isomorphisms $\hx(X/\overline{A}) = \hx(X - A)$ and $\Hx(X/\overline{A}) = \Hx(X - A)$.
\end{Prop}
\begin{Remark}
    Another coarse model for $(X \cminus A)$ can be obtained by the ``coning off'' construction, popular in the study of relatively hyperbolic groups.
    Here instead of taking a quotient by $A$ one attaches a cone $cA$ on $A$, and declares $d(c, a)=1$ for all $a \in A$.
\end{Remark}
\section{Separations and \texorpdfstring{$\Hx[1](X-A;\zz/2)$}{HX1(X-A; Z/2)}}\label{s:separation}

We now consider the coarse cohomology with $\zz/2$ coefficients $\Hx(X-A;\zz/2)$.
It is the cohomology of the cochain complex
\[
\Cx[n](X-A; \zz/2) = \{ \phi \in \C[n](X; \zz/2) \mid \Supp{\phi} \ccap \gD \csubset \gD_A \}.
\]
We will see that $\Hx[1](X-A;\zz/2)$ measures the number of coarse components of the complement of $A$.
Making this precise requires some preparation.

There is a simple interpretation of $1$-cocycles with $\zz/2$-coefficients in terms of subsets of $X$.
The power set $2^{X}$ is a boolean algebra, with multiplication given by the intersection and addition by the symmetric difference.
The complementation map $C \to X-C$ is a free involution on $2^{X}$.
Let $\cS$ be the quotient space, we will think of its elements as unordered pairs $\{C, X-C\}$, and refer to them as \emph{separations} of $X$.
Alternatively, since $X-C=X \symdiff C$, one can think of $\cS$ as a quotient of abelian groups $2^{X}/ \{\emptyset, X\} $.
As an abelian group $2^{X}$ is isomorphic to $\C[0](X; \zz/2) $, the isomorphism is given by taking supports and the map $C \to 1_{C}$.
Thus we get a commutative diagram, where $F$ is the induced isomorphism.
\[
\begin{tikzcd}
    0 \ar[r] & \{\emptyset, X\} \ar[r] \ar[d, equal] & 2^{X} \ar[r] \ar[d, "1_{(\cdot)}", rightharpoonup, shift left=1pt] & \cS \ar[r] \ar[d, "F"]& 0 \\
    0 \ar[r] &\zz/2 \ar[r] & \C[0](X; \zz/2) \ar[r, "d"] \ar[u, "\Supp{\cdot}", rightharpoonup, shift left=1pt] & Z^{1}(X; \zz/2) \ar[r] & 0
\end{tikzcd}
\]
More explicitly, $F(\{C, X-C\})=d(1_{C})$, and its inverse is given by the following lemma.
\begin{Lemma}
    Given a cocycle $\phi \in Z^{1}(X; \zz/2)$ consider the following relation on $X$: $x \sim y$ if $\phi(x, y)=0$.
    This relation is an equivalence relation with at most two equivalence classes, i.e.
    a separation.
\end{Lemma}
\begin{proof}
    Since $\phi$ is a cocycle with $\zz_2$-coefficients, for any $x, y, z \in X$ we have
    \[
    \phi(x, y)+\phi(x, z)+\phi(y, z)=0.
    \]
    This equation immediately implies the transitivity of the relation.
    Taking $z=y=x$ in the equation shows that $\phi(x, x)=0$ and hence the relation is reflexive.
    Taking $z=y$ shows then that $\phi(x, y)=\phi(y, x)$ and hence the relation is symmetric.

    The equation also implies that if neither $y$ nor $z$ is related to $x$, then $y$ and $z$ are related, and therefore there are at most two equivalence classes.
\end{proof}

We now introduce coarse supports into the picture.
Some of the following definitions are inspired by \cite{m18}.
Let $A\subset X$, a subset $C\subset X$ is a \emph{coarse complementary component} of $A$ if $C\ccap (X-C)\csubset A$.
$C$ of $A$ is \emph{shallow} if $C\csubset A$, otherwise it is \emph{deep}.
The \emph{$r$-boundary} of $C$ is
\[
\partial_r C := \{x \in X-C \mid d(x, C) \leq r\}.
\]
The following description follows easily from Lemma~\ref{l:nbds}.
\begin{Lemma}\label{l:complement criterion}
    $C$ is a coarse complementary component of $A$ if and only if $\partial_r(C) \csubset A$ for all $r$.
\end{Lemma}
We will call a collection $\{C_{\ga}\}$ of coarse complementary components \emph{uniform} if for any $r$ there exists $R$ such that $\d_r(C_{\ga}) \subset[R] A$ for all $C_{\ga}$.
Since $\d_r$ of an arbitrary union is contained in the union of $\d_{r}$'s, Lemma~\ref{l:complement criterion} implies
\begin{Lemma}\label{l:union}
    If $\{C_{\ga}\}$ is uniform, then its union $\bigcup C_{\ga}$ is a coarse complementary component of $A$.
\end{Lemma}

Let $\cC_{A}$ denote the collection of all coarse complementary components of $A$.
As before, each $C \in \cC_{A}$ determines a separation $\{C, X-C\} \in \cS$, we will refer to it as a \emph{separation of $X$ with respect to $A$}.
Such a separation is \emph{deep} if both $C$ and $X-C$ are deep, and \emph{shallow} otherwise.
We will say that \emph{$A$ separates $X$} if there exists a deep separation of $X$ with respect to $A$.
Let $\cS_A$ denote the collection of all such separations, and let $\cS\cS_A$ be the subcollection of shallow separations.

By Lemma~\ref{l:union}, $\cC_{A}$ is closed under finite unions, and since it is closed under the complementation, it follows that $\cC_{A}$ is a subalgebra of the boolean algebra $2^{X}$, and $\cS\cS_{A}$ and  $\cS_{A}$ are subgroups of $\cS$.
Let $\cD\cS_A=\cS_A/\cS\cS_A$, its nonzero elements are equivalence classes of deep separations of $X$ with respect to $A$ modulo shallow ones.

Note that $1_{C} \in \Cx[0](X-A; \zz/2)$ precisely when $C$ is shallow, so $F$ maps isomorphically $\cS\cS_{A}$ onto $BX^{1}( X-A; \zz/2)$.
Moreover,
\[
\Supp{d(1_{C}) } \ccap \gD_{X} =(C\times (X-C) \cup (X-C) \times X) \ccap \gD = \gD_{C} \ccap \gD_{X-C},
\]
so $F$ restricts to an isomorphism $F: \cS_{A} \to ZX^{1}( X-A; \zz/2)$.

Thus we have:
\begin{Prop}\label{p:Hx and separation}
    $F$ induces an isomorphism $\cD\cS_A \to \Hx[1](X-A; \zz/2)$.
\end{Prop}

We will need the following observation in Section~\ref{s:PD}.
\begin{Lemma}\label{l:deepsep}
    Suppose $A$ coarsely separates $X$.
    Then any deep coarse complementary component $C$ of $X$ with respect to $A$ contains a deep coarse complementary component $C' \subset C$ such that the separation $\{C', X-C'\}$ is deep.
\end{Lemma}
\begin{proof}
    If $X-C$ is deep we can take $C'=C$, so assume that $X-C$ is shallow.
    By hypothesis, we have a deep separation $\{D, X-D\}$.
    Set $C'=C\cap D$.
    Then both $C'=D-(X-C)$ and its complement $X-C'=(X-D)\cup (X-C)$ are deep.
\end{proof}

The next lemma shows that pairwise disjoint collections of deep coarse complementary components bound from below the dimension of $\Hx[1](X-A; \zz/2)$.
\begin{Lemma}\label{l:dim}
    Suppose $\{C_{\ga}\}$ is a pairwise disjoint collection of deep coarse complementary components of $X$ with respect to $A$.
    Then any proper subcollection of $\{C_{\ga}\}$ maps to a linearly independent subset of $\Hx[1](X-A; \zz/2)$.
\end{Lemma}
\begin{proof}
    Since the coefficients are $Z/2$, we need to show that any finite sum $\sum_{i=1}^{n} C_{\ga_{i}}$ of the elements of the subcollection maps nontrivially.
    Disjointness implies that the sum is just the union, so it is deep, and so is its complement, since it contains a deep element not in the subcollection.
    Therefore the corresponding separation is not in $\cS\cS_{A}$.
\end{proof}

We now specialize to the case of geodesic metrics.
\begin{Prop}\label{p:uniform}
    If $X$ is a geodesic space and $A\subset X$, then for any $r$, the collection of path components of $X-N_r(A)$ is a uniform collection of coarse complementary components.
\end{Prop}
\begin{proof}
    Fix $r$, and let $C$ be a path component of $X-N_r(A)$.
    We will show that for any $R$, $\partial_R(C) \subset[R+r] A$.

    Let $x\in \partial_R(C)$.
    Moreover, assume that $x\notin N_r(A)$, otherwise there is nothing to prove.
    Then there exists $y \in C$ such that $d(x, y)\leq R$.
    Since $x \notin C$ and $y \in C$, $x$ and $y$ are in different path components of $X-N_r(A)$, thus a geodesic between $x$ and $y$ must dip into $N_{r}(A)$: there exists $z$ on the geodesic with $d(z, A) < r$.
    Note that $d(x, z) \leq d(x, y) \leq R$.
    Then by triangle inequality $d(x, A) \leq d(x, z) + d(z, A) < R+r$, as claimed.
\end{proof}
\begin{Prop}\label{p:geo}
    Let $X$ be a geodesic space and $A, C\subset X$.
    The following are equivalent:
    \begin{enumerate}
        \item\label{i:c-set}
        $C$ is a coarse complementary component of $A$.
        \item\label{i:r exists}
        There exist $r$ such that $C-N_r(A)$ is a union of path components of $X-N_r(A)$.
        \item\label{i:any r}
        There exist $r$ such that for any $R\geq r$, $C-N_R(A)$ is a union of path components of $X-N_R(A)$.
    \end{enumerate}
\end{Prop}
\begin{proof}
    $\eqref{i:any r}\Rightarrow \eqref{i:r exists}$ is clear.

    $\eqref{i:r exists}\Rightarrow \eqref{i:c-set}$ follows from Proposition~\ref{p:uniform}.

    $\eqref{i:c-set}\Rightarrow \eqref{i:any r}$ Let $C$ be a coarse complementary component of $A$ in $X$, then there exists $r$ such that $\partial_1(C)\subset N_{r-1}(A)$.
    Let $R \geq r$.
    We claim that $C':=C-N_{R}(A)$ is a union of path components of $X-N_{R}(A)$.

    Suppose $x \in C'$ and $y \in X-N_{R}(A)$ are connected by a path in $X-N_{R}(A)$.
    We need to show that $y \in C'$.
    Choose points $\{x=x_0, \dots, x_n=y\}$ on the path so that $d(x_i, x_{i+1}) < 1$ for all $i$.
    Suppose $x_i\in C'$, then $d(x_{i+1}, A) \geq d(x_i, A)-d(x_i, x_{i+1}) > R-1 \geq r-1$, so $x_{i+1} \notin N_{r-1}(A)$, and therefore $x_{i+1} \notin \partial_1(C)$.
    On the other hand, $d(x_{i+1}, C) \leq d(x_i, x_{i+1}) < 1$.
    Combining these together we get $x_{i+1}\in C'$.
    So by induction $y \in C'$.
\end{proof}
\begin{Theorem}\label{t:geo}
    Let $X$ be a geodesic space and $A\subset X$, and suppose $k=\dim_{\zz/2} \Hx[1](X-A; \zz/2)$ is finite.
    Then, for any $r$ the union of shallow path components of $X-N_r(A)$ is shallow: they all are contained in $N_{R}(A)$ for some $R$.
    Moreover, there exist $r$ such that for any $R\geq r$, $X-N_R(A)$ has exactly $k+1$ deep path components.
\end{Theorem}
\begin{proof}
    If, for some $r$, the union of shallow path components is not shallow, then there is a sequence $\{C_i\}$ of shallow path components of $X-N_r(A)$ such that $C_i$ is not contained in $N_i(A)$ for all $i\in \nn$.
    Then the union of any subsequence is a deep complementary component, and since there are infinitely many subsequences whose unions are pairwise disjoint, $\dim_{\zz/2} \Hx[1](X-A; \zz/2)$ is infinite by Lemma~\ref{l:dim}.
    This proves the first claim.

    To prove the second, we observe that again by Lemma~\ref{l:dim} for any $R$, there exists at most $k+1$ deep path components in $X-N_R(A)$.
    For each of $2^{k}$ elements of $\Hx[1](X-A; \zz/2)$ we pick a separation which maps onto it, and take both of its coarse complementary components.
    So we get $2^{k+1}$ distinct coarse complementary components.
    By Proposition~\ref{p:geo} there exits $r$ such that for any $R\geq r$ these coarse complementary components with $N_{R}(A)$ removed are unions of path components of $X-N_{R}(A)$.
    Next we remove all shallow path components to obtain $2^{k+1}$ coarse complementary components, each of which is a union of deep path components of $X-N_{R}(A)$.
    Since any union of shallow path components is shallow, the removals do not affect the image in cohomology, and therefore all these unions are still distinct.
    Hence, $X-N_{R}(A)$ has at least $k+1$ deep path components.
    This finishes the proof.
\end{proof}

\section{Alexander Duality}\label{s:PD}
From now on, we will be mostly interested in integral (co)homology, so the omitted coefficients are $\zz$.
\begin{Definition}\label{def pdn}
    A metric space $X$ is a \emph{coarse Poincaré duality space of formal dimension $n$} (\emph{coarse PD($n$) space} for short), if there exist chain maps $p: \C(X) \to \cx[n-*](X)$ and $q: \cx[n-*](X) \to \C(X)$, so that $pq$ and $qp$ are chain homotopic to identities via chain homotopies $G:\cx(X) \to \cx[*+1](X)$ and $F: \C(X) \to \C[*-1](X)$ which are controlled:
    \begin{align*}
        \forall \phi\in \C(X) &\qquad \Supp{ p(\phi)} \csubset \Supp{\phi}, \\
        \forall \phi\in \C(X) &\quad \Supp{ F(\phi)} \ccap \gD \csubset \Supp{\phi} , \\
        \forall c \in \cx(X) &\quad \Supp{q(c)} \ccap \gD \csubset \Supp{c}, \\
        \forall c\in \cx(X) &\qquad \Supp{ G(c)} \csubset \Supp{c}.
    \end{align*}
\end{Definition}

The next lemma shows that the maps in the definition of a coarse PD($n$) space restrict to maps between $\Cx(X-A) $ and $\cx(A\subset X)$.
\begin{Lemma}\label{l:res}
    \setcounter{equation}{0}
    \begin{align}\label{i:p}
        p (\Cx(X-A)) &\subset \cx[n-*] (A \subset X) & p (\Cx(A \subset X)) &\subset \cx[n-*] (X-A)\\\label{i:F}
        F(\Cx(X-A)) &\subset \Cx[*-1](X-A) & F(\Cx(A \subset X)) &\subset \Cx[*-1](A \subset X) \\\label{i:q}
        q (\cx(A \subset X)) &\subset \Cx[n-*](X-A) & q(\cx(X-A)) &\subset \Cx[n-*](A \subset X) \\\label{i:G}
        G(\cx[*](A \subset X)) &\subset \cx[*+1](A \subset X) & G(\cx[*](X-A)) &\subset \cx[*+1](X-A)
    \end{align}
\end{Lemma}
\begin{proof}
    The proof is a routine exercise in coarse set theory.

    \eqref{i:p} By assumption on $p$, $\Supp{p(\phi)} \csubset \Supp{\phi} $.
    Since $p$ maps into $\cx[n-*](X)$, $\Supp{p(\phi)} \csubset \gD $.
    If $\phi \in \Cx(X-A)$, then $\Supp{\phi} \ccap \gD \csubset \gD_A$.
    Combining these using Lemma~\ref{l:subcap} we get $\Supp{p(\phi)} \csubset \gD_A $ and thus $p(\phi) \in \cx[n-*](A \subset X)$.

    If $\phi \in \Cx(A \subset X)$, then $\Supp{\phi} \ccap \gD_{A} \csubset *$.
    So $\Supp{p(\phi)} \csubset \gD_{X} \cminus \gD_A $ and thus $p(\phi) \in \cx[n-*](X-A)$.

    \eqref{i:F} By assumption on $F$, $\Supp{ F(\phi)} \ccap \gD \csubset \Supp{\phi}$.
    If $\phi \in \Cx(X-A)$ then $\Supp{\phi} \ccap \gD \csubset \gD_A$.
    Combining these we get $\Supp{F(\phi)} \ccap \gD \csubset \gD_A $.
    So $F(\phi) \in \Cx[*-1](X-A)$.

    If $\phi \in \Cx(A \subset X)$, then $\Supp{\phi} \ccap \gD_{A} \csubset *$.
    It follows that $\Supp{ F(\phi)} \ccap \gD_{A} \csubset *$ and $F(\phi) \in \Cx[*-1](X-A)$.

    \eqref{i:q} Let $c \in \cx(A \subset X)$, so $\Supp{c} \csubset \gD_{A}$.
    By assumption on $q$, $\Supp{q(c)}\ccap \gD \csubset \Supp{c} $.
    Thus $\Supp{q(c)} \csubset \gD_{A} $.
    Hence $q(c) \in \Cx[n-*](X-A)$.

    Let $c \in \Cx(X-A)$ then $\Supp{c} \ccap \gD_A \csubset *$ Hence by the property of $q$, $\Supp{q(c)}\ccap \gD_A \csubset \Supp{c}$.
    By Lemma \ref{l:subcap}, we have $q(c)\in \Cx[n-*](A\subset X)$.

    \eqref{i:G} Let $c \in \cx(A \subset X)$, so $\Supp{c} \csubset \gD_{A}$.
    By assumption on $G$, $\Supp{G(c)} \csubset \Supp{c} $.
    Thus $\Supp{G(c)} \csubset \gD_A $.
    Hence $G(c) \in \cx[*+1](A\subset X)$.

    Let $c\in \Cx(X-A)$, so $\Supp{c} \ccap \gD_A \csubset *$.
    By assumption on $G$, $\Supp{G(c)} \csubset \Supp{c} $.
    Combining these, we get $\Supp{G(c)} \ccap \gD_A \csubset *$.
\end{proof}
\begin{Theorem}[Coarse Alexander Duality]\label{t:pd}
    If $X$ is a coarse PD($n$) space, then for any $A\subset X$ and any finitely generated abelian group $G$
    \begin{align*}
        \Hx[k](X-A; G) &\cong \hx[n-k](A; G), \\
        \hx[k](X-A; G) &\cong \Hx[n-k](A; G).
    \end{align*}
\end{Theorem}
\begin{proof}
    By Lemmas \ref{l:res} and \ref{l:equiv1}, $p\otimes G$ and $q \otimes G$ induce maps between $\Hx(X-A; G)$ and $\hx[n-*](A; G )$, and $\Hx(A; G)$ and $\hx[n-*](X-A; G )$, which are inverses of each other.
\end{proof}
One immediate corollary of the above theorem is the following computation of coarse (co)homology of coarse PD($n$) spaces.
\begin{Corollary}\label{c:HX of PD(n)}
    Suppose $X$ is a coarse PD($n$) space, and $G$ be a finitely generated abelian group.
    Then:
    \begin{enumerate}
        \item\label{i:hx of PD(n)}

        $ \Hx[*](X;G)=\hx[*](X; G)=
        \begin{cases}
            G & *=n,\\
            0 & \text{otherwise}.
        \end{cases}
        $

        \noindent Furthermore, any cycle representing a non-zero element of $\Hx[n](X; G)$ coarsely contains $\gD_{X}$ in its support.
        \item\label{i:hx of subset}
        If $A\subset X$ and $X$ is not coarsely contained in $A$, then $\Hx[n](A; G)=0$.
    \end{enumerate}
\end{Corollary}
\begin{proof}
    \eqref{i:hx of PD(n)} Taking $A=\emptyset$ in Theorem~\ref{t:pd}, we have
    \[
    \Hx(X; G)=\Hx(X-\emptyset; G)=\hx[n-*](\emptyset; G)=
    \begin{cases}
        G & *=n,\\
        0 & \text{otherwise.}
    \end{cases}
    \]
    A similar argument computes homology.

    For the second part, consider the map $q:\cx[n](X)\to \C[0](X)$ from the definition of a coarse PD($n$) space.
    If an $n$-cycle $c$ is nonzero in $\hx[n](X; G)$, then its dual $0$-cocycle $\phi=q\otimes G(c)$ is a nonzero constant function $\phi: X \to G$, in particular $\Supp{\phi}=\gD$.
    By definitions $\Supp{\phi} \ccap \gD\csubset \Supp{c}$ and $\Supp{c} \csubset \gD $, thus $\Supp{c} \ceq \gD $.

    \eqref{i:hx of subset} Since $0$-cocycles on $X$ are constant functions, $X\not \csubset A$ implies that the only $0$-cocycle in $\Cx[0](X-A; G)$ is $0$.
    Hence, $\Hx[n](A; G)=\Hx[0](X-A; G) =0$.
\end{proof}

Combining Proposition~\ref{p:Hx and separation} and the coarse Alexander duality shows that deep separations of a coarse PD($n$) space are determined by the $(n-1)^{\text{st}}$ coarse homology of the separating subspace:
\begin{Corollary}\label{c:ds}
    Let $X$ be a coarse PD($n$) space and $A\subset X$.
    Then
    \[
    \cD\cS_A=\hx[n-1](A; \zz/2).
    \]
\end{Corollary}

As an application, we prove the following coarse version of the Jordan separation theorem, which combined with Theorem~\ref{t:geo} in the geodesic case, essentially recovers the coarse Jordan separation theorem of Kapovich and Kleiner \cite{kk05}*{Corollary 7.8}, see also~\cite{p07}*{Proposition 1}.
\begin{Corollary}[Coarse Jordan separation] Let $f:Y \to X$ be a coarse embedding, where $X$ is a coarse PD($n$) space.
    Then:
    \begin{enumerate}
        \item If $Y$ is a coarse PD($n-1$) space, then modulo shallow separations, there is a unique deep coarse separation of $X$ with respect to $f(Y)$.

        \noindent Moreover, if $C$ is a deep coarse complementary component of $X$ with respect to $f(Y)$, then $f(Y)\csubset C$.

        \item If $Y\subset Z$, where $Z$ is a coarse PD($n-1$) space, and $Z$ is not coarsely contained in $Y$, then $f(Y)$ does not coarsely separate $X$.
    \end{enumerate}
\end{Corollary}
\begin{proof}
    Since $f$ is a coarse equivalence between $Y$ and $f(Y)$, the separation claims follow immediately from Corollaries~\ref{c:HX of PD(n)} and \ref{c:ds}.

    To prove the coarse containment claim, note that the claim for a subset of $C$ implies the claim for $C$, so by Lemma~\ref{l:deepsep} it is enough to consider the case where both $C$ and $X-C$ are deep.

    In this case $[d(1_C)] \neq 0$ in $\Hx[1](X-f(Y); \zz/2)$ and hence its dual $[p(d(1_C))] \neq 0$ in $\hx[n-1](f(Y) \subset X; \zz/2)$.
    By construction of Lemma~\ref{l:equiv1} the cycle $p(d(1_C))$ is homologous to a cycle $c$ in $\cx[n-1](f(Y); \zz/2)$ with $\Supp{c} \ceq \Supp{p(d(1_C))}$.
    Since $\Supp{d(1_C)} \ccap \gD_{X }\csubset \gD_{C}$, arguing like in Lemma~\ref{l:res}\eqref{i:p} shows that $\Supp{p(d(1_C))}\csubset C$.
    By Corollary~\ref{c:HX of PD(n)} we have $f(Y)\csubset \Supp{c}$ and we conclude $f(Y)\csubset C$.
\end{proof}

\section{Products}\label{s:products}

Our next goal is to obtain a characterization of coarse PD($n$) spaces in terms of existence of a certain pair of (co)homology classes (Theorem~\ref{t:op}).
Our proof mostly follows the proof of a similar characterization of the usual homology manifolds by McCrory~\cite{m77}, we just need to keep careful track of supports.
In this section we gather the required information about various products.
Here we closely follow~\cite{s81a}*{Chapter 6}.
Our main observation is that certain slant and cap product maps between $\cx(X)$ and $\C(X)$ satisfy the support conditions of the maps $p$ and $q$ in the definition of a coarse PD($n$) space (Lemma~\ref{l:maps}).

To define product maps between (co)chain complexes, it is convenient to work with $\cf(X)\otimes \cf(Y)$, rather than with $\cf(X\times Y)$.
On the other hand, certain homotopies are easier to describe in $\cf(X\times Y)$.
So we will need to transfer these homotopies in $\cf(X\times Y)$ to homotopies in $\cf(X)\otimes \cf(Y)$.
To have control over the support of the homotopies we need to fix a particular choice of chain homotopy equivalence between $\cf(X\times Y)$ and $\cf(X) \otimes \cf(Y)$.

To simplify notation, denote $\cf(X) \otimes \cf(Y)$ by $\cf(X \otimes Y)$ and $\Hom(\cf(X) \otimes \cf(Y), \zz)$ by $\C(X \otimes Y)$.

For a simplex $\gs =(x_0, \dots, x_i, \dots, x_k)$ let ${}_i\gs=(x_0, \dots, x_i)$ and $\gs_{k-i}=(x_i, \dots, x_k)$ denote its front and back faces respectively.
The Alexander--Whitney map is given by:
\begin{align*}
    A: \cf(X \times Y) \to \cf(X) \otimes \cf(Y) \\
    \gr^{k} \mapsto \sum_{i} {}_i p_{X}(\gr) \otimes p_{Y}(\gr)_{k-i}.
\end{align*}
Here $p_{X}$ and $p_{Y}$ denote the projections to $X$ and $Y$, respectively: if $\gr^{k} = (x_{0}, y_{0} , \dots, x_{k}, y_{k} )$, $p_{X}(\gr):=(x_0, \dots, x_k) $ and $p_{Y}(\gr):=(y_0, \dots, y_k)$.

To get a map in the other direction, take two simplices $\sigma=(x_0, \dots , x_k)$ and $\tau=(y_0, \dots , y_l)$.
Arrange pairs $(x_{i}, y_{j})$ as the vertices of an $k\times l$ rectangular grid in $\rr^2$.
To a path $\gamma$ in this grid starting at $(x_{0}, y_{0})$ and ending at $(x_{k}, y_{l})$, always moving either north or east, associate a simplex $\rho_{\gamma}$ in $X\times Y$, obtained by reading the labels along the path.
Define the simplicial cross product
\begin{align*}
    S :\cf(X \otimes Y) &\to \cf(X\times Y) \\
    \gs^{k} \otimes \gt^{l} &\mapsto \sum_{\gamma}(-1)^{n(\gamma)}\, \rho_\gamma,
\end{align*}
where $n(\gamma)$ is the number of squares in the grid lying below the path $\gamma$.
Define the support $\Supp{\gs \otimes \gt}$ of $\gs\otimes \gt$ to be the element of $ (X \times Y)^{(k+1)(l+1)}$ whose $(i, j)$ coordinate is $(x_i, y_j)$.
Let $\clf(X\otimes Y)$ consists of all integral chains of the form $c=\sum a_{\gs, \tau} \gs\otimes \tau $ with the following support condition: there are only finitely many products of simplices that have a vertex in any given bounded set in $X\times Y$.
We will be interested in the following subcomplex of $\clf(X\otimes Y)$:
\[
\cx(X\otimes Y):=\{c\in \clf(X\otimes Y)\mid \Supp{c}\csubset \gD_{X\times Y}\}.
\]
\begin{Lemma}\label{l:AW}
    The maps $A$ and $S$ are coarsely support preserving chain homotopy equivalences between $\cx(X \times Y)$ and $\cx(X \otimes Y)$.
\end{Lemma}
\begin{proof}
    Note that $\Supp{S(\gs \otimes \gt)} \subset \Supp{\gs \times \gt}$ and $\Supp{A(\gr)} \subset \Supp{ p_{X}(\gr) \times p_{Y}(\gr)}$.
    The acyclic models produce chain homotopies $H: \cf(X \otimes Y) \to \cf[*+1](X \otimes Y) $ and $H': \cf(X \times Y) \to \cf[*+1](X \times Y) $ between their compositions and identities, which satisfy the same support conditions.
    (The point is that the acyclic models use the same vertices.) 
    Since $ \Supp{ p_{X}(\gr) \times p_{Y}(\gr)} \subset[\diam \gr] \Supp{\gr}$, all these maps extend to coarsely support preserving maps between coarse chain complexes.
\end{proof}
\begin{Remark}
    A similar statement, in a slightly different language, is proved by Hair in \cite{h10}, where an explicit formulas for the homotopies are given.
\end{Remark}

We now consider the case $X=Y$.
Let $\gd$ denote the diagonal map from $X$ to $X\times X$, i.e.
$\gd(x)=(x, x)$ for all $x\in X$.

Let $T:X\times X\to X\times X$ be the involution map switching factors.
It induces the chain involutions $T_{*}$ on $\cx(X\times X)$ and $\cx(X\otimes X)$, the latter is given by $\gs^{k} \otimes \gt^{l} \mapsto (-1)^{kl} \gt^{l} \otimes \gs^{k} $.

Define a chain homotopy $D_{*}$ between $T_*$ and $id$ in $\cf(X\times X)$ by
\[
D_{*}(x_0, \dots , x_m)=\sum (-1)^k(x_0, \dots , x_k, T(x_k), \dots , T(x_m)).
\]
Since metric neighborhoods of $\gd(X)$ in $X \times X$ are invariant under $T$, it follows that $D$ extends to coarsely support preserving maps between coarse chain complexes $\cx(\gd(X) \subset X \times X)$.

By restricting the map $S$ of Lemma~\ref{l:AW} to $\cx(\gd(X)\subset X\times X)$, we get that $\cx(\gd(X)\subset X\times X)$ is coarse chain homotopy equivalent to the following complex
\[
\cx(\delta(X)\subset X\otimes X):=\{c\in \clf(X\otimes X)\mid \Supp{c}\csubset \gD_{\gd(X)}\}.
\]

Note that $S$ commutes with $T_{*}$: $ST_{*}=T_{*}S$.
So we can concatenate the homotopies from Lemma~\ref{l:AW} with $D$ to obtain a chain homotopy $D^{T}$ between $T_*$ and $id$ in $\cx(\gd(X) \subset X\otimes X)$:
\[
T_{*} \he[H_{*}T_{*}] AST_{*} = AT_{*}S \he[AD_{*}S] AS \he[H'_{*}] id.
\]
In other words, $D^{T}=H'+ ADS - HT_{*}$ and $T_{*} - id= \d D^{T} + D^{T} \d$.

We summarize these observations in the following lemma.
\begin{Lemma}\label{l:T}
    \leavevmode
    \begin{enumerate}
        \item The homotopy $D^{T}_{*}: \cf(X\otimes X) \to \cf(X\otimes X) $ satisfies
        \[
        \Supp{D^{T}_{*}(c)} \ccap \gD_{\gd(X)}\ceq \Supp{c} \ccap \gD_{\gd(X)}.
        \]
        \item\label{i:C}
        The homotopy $D_{T}^{*}: \C(X\otimes X) \to \C(X\otimes X) $ satisfies
        \[
        \Supp{D_{T}^{*}(\phi)} \ccap \gD_{\gd(X)}\ceq \Supp{\phi} \ccap \gD_{\gd(X)}.
        \]
        \item The map $T_{*}: \cx(\gd(X)\subset X\otimes X) \to \cx(\gd(X)\subset X\otimes X) $ is controlled chain homotopic to the identity.
    \end{enumerate}
\end{Lemma}
The Alexander--Whitney diagonal approximation is the composition:
\[
\cf(X) \xrightarrow{\gd_{*}}\cf(X \times X) \xrightarrow{A} \cf(X \otimes X).
\]

Its tensor product with itself is used for the cap and cup products in $X\otimes X$:
\begin{align*}
    \gd^{A}_{*}: \cf(X\otimes X) &\to \cf(X\otimes X \otimes X \otimes X) \\
    \gs^{k} \otimes \gt^{l} &\mapsto \sum_{i, j} (-1)^{j(k-i)} ( {}_i\gs \otimes {}_{j}\gt )\otimes (\gs_{k-i} \otimes \gt_{l-j}).
\end{align*}

We now recall chain versions of cup, cap, and slant products from~\cite{s81a}.
\subsection{Cup product} 
Suppose $\phi, \psi \in \C(X\otimes X)$.
Then the cross product $\phi\times \psi\in \C(X\otimes X\otimes X\otimes X)$ is defined by
\[
(\phi\times \psi)(\gs\otimes \tau\otimes \gs'\otimes \tau'):=\phi(\gs\otimes \tau)\psi(\gs'\otimes \tau').
\]
Let $\gd^*_A$ be the dual of the map $\gd_*^A$.
The cup product $\phi\cupp \psi \in \C(X\otimes X)$ is defined to be
\[
\gd_A^*(\phi\times \psi).
\]

\subsection{ Cap product} Suppose $\psi\in \C[n](X)$ and $\gs^k\in \cf[k](X)$ is a $k$-simplex.
The cap product is defined by
\[
\psi \capp \gs=\psi(\gs_{n})(_{k-n}\gs).
\]
Suppose, $\phi \in \C[n](X\otimes X)$ and $\gs^k\otimes \tau^l\in \cf(X\otimes X)$.
Then the cap product is defined by
\[
\phi \capp (\gs^{k} \otimes \gt^{l})=\sum_{i+j=n} (-1)^{i(l-j)} ( \phi(\gs_{i} \otimes \gt_{j})) ({}_{k-i}\gs \otimes {}_{l-j}\gt ).
\]

\subsection{ Slant product} Suppose $\phi\in \C[n](X\otimes X)$ and $\gs^k\in \cf[k](X)$ is a $k$-simplex.
The slant product $\phi/\gs\in \C[n-k](X)$ is defined by
\[
(\phi/\gs^k)(\gt^{n-k}) = \phi(\gs \otimes \gt).
\]

We first observe the following.
\begin{Lemma}\label{l:diagapp}
    Suppose $\phi\in \C(X\otimes X\otimes X\otimes X)$ and $\delta:X\times X\to X\times X\times X\times X$ is the diagonal map.
    Then $\Supp{ \gd_{A}^{*} (\phi)} \ccap \gD_{X\times X} \csubset \gd^{-1} (\Supp{\phi}\ccap \gD_{\gd(X\times X)})$.
\end{Lemma}
\begin{proof}

    Suppose $\Supp{\gs^k\otimes \tau^l}\in \Supp{\gd_A^*(\phi)}\cap N_r(\gD_{X\times X})$.
    It is enough to show that $\Supp{\gs^k\otimes \gt^l}\in \gd^{-1}(N_r(\Supp{\phi })\cap N_r(\gD_{\gd(X\times X)}))$.

    The assumption $\Supp{\gs^k\otimes \gt^l}\in \Supp{\gd_A^*(\phi)}$ implies that
    \[
    \Supp{{}_i\gs \otimes {}_{j}\gt \otimes \gs_{k-i} \otimes \gt_{l-j}}\in \Supp{\phi}
    \]
    for some $i, j$.

    The assumption $\Supp{\gs^k\otimes \gt^l}\in N_r(\gD_{X\times X})$ implies that
    \[
    \Supp{\gs^k\otimes \gt^l}\in \delta^{-1}(N_r(\gD_{\delta(X\times X)})),
    \]
    and also that
    \[
    \delta(\Supp{\gs^k \otimes\tau^l})\subset N_r(\Supp{{}_i\gs \otimes {}_{j}\gt \otimes \gs_{k-i}\otimes \gt_{l-j} }),
    \]
    and hence
    \[
    \Supp{\gs^k\otimes \gt^l}\in \gd^{-1}(N_r(\Supp{\phi})).
    \]
    Combining these we obtain
    \[
    \Supp{\gs^k\otimes \gt^l}\in \gd^{-1}(N_r(\Supp{\phi }))\cap \delta^{-1}(N_r(\gD_{\gd(X\times X)}))=\gd^{-1}(N_r(\Supp{\phi })\cap N_r(\gD_{\gd(X\times X)})).
    \]
    This finishes the proof.
\end{proof}

To state our next lemma, we need the following cochain complex
\[
\Cx(X\otimes X-\gd(X)):=\{\phi\in \C(X\otimes X)\mid \Supp{\phi}\ccap \gD_{X\times X} \csubset \gD_{\gd(X)}\}.
\]
Note that this complex is chain homotopy equivalent to $\Cx(X\times X-\gd(X))$ via coarsely support preserving homotopies.

Let $p_{1*}, p_{2*}:\cf(X\otimes X)\to \cf(X)$ be the projection maps: $p_{1*}(\gs\otimes \tau)=\gs$ and $p_{2*}(\gs\otimes \tau)=\tau$.
They induce maps $p^{*}_{1}, p^{*}_{2}: \C(X) \to \C(X\otimes X)$.
For the rest of the paper, we set $\ge'_{n}=(-1)^{n(n+1)/2}$ and $\ge_{n}=(-1)^{n}$.
\begin{Lemma}\label{l:maps}
    Let $\phi \in \Cx[n](X\otimes X - \gd(X))$ be a cocycle and let $c\in \cx[n](X)$ be a cycle.
    Then:
    \begin{enumerate}

        \item\label{i:slant}
        The slant product map, $\ge'_{n-k-1} \phi/ \cdot : \cx[k](X) \to \C[n-k](X)$ is a chain map and $\Supp{\phi/a} \ccap \gD_X \csubset \Supp{a}$ for all $a\in \cx(X)$.

        \item\label{i:cap}
        The cap product map $ \ge'_{n-k} \cdot \capp c :\C[k](X) \to \cx[n-k](X)$ is a chain map and $\Supp{\psi \capp c} \csubset \Supp{\psi}$ for all $\psi \in \C[k](X)$.

        \item\label{i:capx}
        The chain maps $\cx(X) \to \cx(\gd(X) \subset X \otimes X)$ given by $a \mapsto \ge_{nk} \phi \capp (c \otimes a )$ and $a \mapsto T^{*}\phi \capp (a \otimes c)$ are controlled chain homotopic: the chain homotopy $D$ satisfies $\Supp{D(a)} \csubset \gd(\Supp{a} )$.

        \item\label{i:cupx}

        The chain maps $\C(X) \to \Cx(X \otimes X - \gd(X))$ given by $\psi \mapsto \ge_{nk} \phi \cupp p_{2}^{*} ( \psi ) $ and $\psi \mapsto p_{1}^{*} ( \psi ) \cupp \phi $ are controlled chain homotopic: the chain homotopy $D$ satisfies $\Supp{D(\psi)} \ccap \gD_{X\times X} \csubset \gd(\Supp{\psi} ) \ccap \gD_{\gd(X)}$.
    \end{enumerate}
\end{Lemma}
\begin{proof}

    \eqref{i:slant} Since $a\in \cx(X)$ and $\phi \in \Cx[n](X\otimes X - \gd(X))$, it follows that $(\phi/a)(\gs)$ is well-defined and $\Supp{\phi/a} \ccap \gD_X \csubset \Supp{a}$.
    The choice of signs together with $d\phi=0$ shows that $\ge'_{n-k-1} \phi/\cdot$ is indeed a chain map:
    \[
    d(\ge'_{n-k-1} \phi/ a) = (-1)^{n-k} \ge'_{n-k-1} \phi/ \d a = \ge'_{n-k} \phi/ \d a.
    \]

    \eqref{i:cap} Since $c \in \cx[n](X)$, $ \psi \capp c$ is a well-defined element of $\cx[n-k](X)$ and $\Supp{\psi \capp c} \csubset \Supp{\psi}$.
    Again the choice of signs together with $\d c=0$ shows that $\ge'_{n-k} \cdot \capp c $ is a chain map:
    \[
    \d (\ge'_{n-k} \psi \capp c ) = (-1)^{n-k} \ge'_{n-k} d \psi \capp c = \ge'_{n-k-1} d \psi \capp c.
    \]

    \eqref{i:capx} It follows as before that both maps are well-defined chain maps.
    For $a \in \cx[k](X)$ we have
    \[
        T^*\phi \capp (a \otimes c) = \ge_{nk} T^{*}\phi\capp T_{*}(c \otimes a) =\ge_{nk} T_*(T^{*2}\phi\capp (c \otimes a)) = \ge_{nk} T_*(\phi\capp (c \otimes a)).
    \]    
   
    By assumptions on $c$ and $a$ we have $\Supp{c} \times \Supp{a} \csubset \gD_{X\times X}$ and $\Supp{c} \times \Supp{a} \ccap \gD_{\gd(X)}\csubset \gd(\Supp{a})$.
    Therefore, by assumptions on $\phi$,
    \[
    \Supp{\phi\capp (c \otimes a)} \csubset \Supp{\phi} \ccap (\Supp{c} \times \Supp{a}) \csubset \Supp{\phi} \ccap \gD_{X\times X} \ccap (\Supp{c} \times \Supp{a}) \csubset \gd(\Supp{a}).
    \]
    The claim follows, since the homotopy between $T_{*}$ and $id$ coarsely preserves supports by Lemma~\ref{l:T}.

    \eqref{i:cupx} Suppose $\psi \in \C(X)$.
    Then $p_i^*(\psi)\in \C(X\otimes X)$.
    We have
    \begin{align*}
        p_{1}^{*} ( \psi ) \cupp \phi = \gd_{A}^{*}(p_{1}^{*} ( \psi ) \times \phi ) &= \gd_{A}^{*}T^{*}_{X\times X}( \ge_{nk}\phi \times p_{1}^{*} ( \psi ) )\\
        &=\gd_{A}^{*}T^{*}_{X\times X}( \ge_{nk}\phi \times T^{*}p_{2}^{*} ( \psi ) ).
    \end{align*}

    Since $ \cdot \otimes 1$ and $ \ge_{nk} \phi \times \cdot $ are chain maps, the concatenation of the homotopies of $T^{*}$ and $T^{*}_{X\times X} $ to identities gives a chain homotopy between the maps in the claim:
    \[
    D(\psi):= \gd_{A}^{*} T^{*}_{X\times X}(\ge_{k-1} \phi \times D_{T}^{*}p_{2}^{*} ( \psi ) ) + \gd_{A}^{*}D_{T_{X\times X}}^{*} (\ge_{nk} \phi \times p_{2}^{*} ( \psi ) ).
    \]

    For the support of the first term we have:
    \begin{align*}
        \MoveEqLeft[3] \SupP{ \gd_{A}^{*}T^{*}_{X\times X}( \phi \times D_{T}^{*}p_{2}^{*} ( \psi ) ) } \ccap \gD_{X\times X} \\
        &\csubset \gd^{-1} \bigl( \Supp{T^{*}_{X\times X} ( \phi \times D_{T}^{*}p_{2}^{*} ( \psi ) )} \ccap \gD_{\gd(X\times X)} \bigr) &\text{by Lemma~\ref{l:diagapp}}\\
        &\csubset \gd^{-1} \bigl( \Supp{ D_{T}^{*}p_{2}^{*} ( \psi ) \times \phi } \ccap \gD_{\gd(X\times X)} \bigr)\\
        &\csubset \gd^{-1} \bigl( \Supp{D_{T}^{*}p_{2}^{*} ( \psi ) } \times \Supp{\phi} \ccap \gD_{\gd(X\times X)} \bigr) \\
        &\csubset \Supp{D_{T}^{*}p_{2}^{*} ( \psi ) } \ccap \Supp{\phi} \ccap \gD_{X\times X} \\
        &\csubset \Supp{D_{T}^{*}p_{2}^{*} ( \psi ) } \ccap \gD_{\gd(X)} &\text{by assumption on $\phi$} \\
        &\csubset \Supp{p_{2}^{*} ( \psi ) } \ccap \gD_{\gd(X)} &\text{by Lemma~\ref{l:T}\eqref{i:C}} \\
        & \ceq \gd( \Supp{\psi } \ccap \gD_{X} ).
    \end{align*}

    Similarly, for the support of the second term we have:
    \begin{align*}
        \MoveEqLeft[3] \SupP{ \gd_{A}^{*} D_{T_{X\times X}}^{*} \bigl(\phi \times p_{2}^{*} ( \psi ) \bigr) } \ccap \gD_{X\times X} \\
        &\csubset \gd^{-1} \bigl( \Supp{ D_{T_{X\times X}}^{*} (\phi \times p_{2}^{*} ( \psi ) ) } \ccap \gD_{\gd(X\times X)} \bigr) &\text{by Lemma~\ref{l:diagapp}} \\
        &\csubset \gd^{-1} \bigl( \Supp{ \phi \times p_{2}^{*} ( \psi ) } \ccap \gD_{\gd(X\times X)} \bigr) &\text{by Lemma~\ref{l:T}\eqref{i:C}} \\
        &\csubset \gd^{-1} \bigl( \Supp{ \phi } \times \Supp{p_{2}^{*} ( \psi ) } \ccap \gD_{\gd(X\times X)} \bigr) \\
        &\csubset \Supp{ \phi } \ccap \Supp{p_{2}^{*} ( \psi ) } \ccap \gD_{X\times X} \\
        &\csubset \Supp{p_{2}^{*} ( \psi ) } \ccap \gD_{\gd(X)} &\text{by assumption on $\phi$} \\
        &\ceq \gd(\Supp{\psi } \ccap \gD_{X}).
    \end{align*}
\end{proof}

The next three formulas tell how cup, cap, and slant products interact with each other on the chain level.
\begin{Lemma}
    \cite{s81a}*{6.1.4-6}\label{l:Spanier 123}
    Let $\gs$ and $\gt$ be simplices in $X$ and let $\phi \in \C(X \otimes X)$, and $\psi \in \C(X)$.
    Then:
    \begin{enumerate}
        \item\label{e:scup}
        $\psi\cupp(\phi/\gt) =( p_{1}^{*} ( \psi ) \cupp \phi ]/\gt$.
        \item\label{e:cupcap}
        $\phi/(\psi \capp \tau)= (\phi \cupp p_{2}^{*} ( \psi ) )/\gt$.
        \item\label{e:cap}
        $(\phi /\sigma)\capp \tau=p_{1*}(\phi \capp (\gt \otimes \gs))$.
    \end{enumerate}
\end{Lemma}

\section{Algebraic Poincaré Duality}\label{s:alg pd}

In this section we find a criterion for a space to be a coarse PD($n$) space in terms of existence of a certain pair of (co)homology classes.
\begin{Definition}
    A pair of classes $U\in \Hx[n](X\otimes X - \gd(X))$ and $[X] \in \hx[n](X)$ is an \emph{orientation pair} if $U/[X]=1 \in \Hx[0](X-X)$.
    We will call $U$ a \emph{diagonal class} and $[X]$ a \emph {fundamental class}.
\end{Definition}
\begin{Remark}
    This condition is stronger than it looks.
    Since the only representative of $1 \in \Hx[0](X-X)$ is the constant function, the above equation is satisfied on the (co)chain level, independent of the choice of representatives for $U$ and $[X]$.
    It also follows that the support of any representative of $[X]$ is coarsely equal to $\gD$.
\end{Remark}
\begin{Lemma}\label{l:opair}
    Let $U\in \Hx[n](X\otimes X - \gd(X))$ and $[X] \in \hx[n](X)$ be a pair of classes and $p_{1*}:\cf(X\otimes X)\to \cf(X)$ is the projection map: $p_{1*}(\gs\otimes \tau)=\gs$.
    The following conditions are equivalent.
    \begin{enumerate}
        \item\label{i:op} $(U, [X])$ is an orientation pair.
        \item\label{i:X} $[X]\ne 0$ and $p_{1*}(U \capp ([X] \otimes [X]))=[X]$.
        \item\label{i:x} For any $x \in X$, $U( x \otimes [X] ) =1$.
    \end{enumerate}
\end{Lemma}
\begin{proof}
    \eqref{i:op}$\Leftrightarrow$\eqref{i:X}.
    By Lemma~\ref{l:Spanier 123}\eqref{e:cap}, $ p_{1*}(U \capp ([X]\otimes [X])) = (U/[X]) \capp [X] $. 
    For nonzero $[X]$ this equals to $[X]$ if and only if $U/[X]=1$.

    \eqref{i:op}$\Leftrightarrow$\eqref{i:x}  follows immediately from the formula $(U/[X])(x) = U(x \otimes [X])$.
\end{proof}
\begin{Theorem}\label{t:op}
    $X$ is a coarse $PD(n)$ space if and only if it has an orientation pair.
\end{Theorem}
\begin{proof}
    Let $X$ be a coarse $PD(n)$ space and let $p$ and $q$ be the associated homotopy equivalences.
    Set $c = p(1) \in \cx[n] (X)$ where $1 \in \Cx[0](X-X)$, and define $\phi \in \C(X\otimes X)$ by the rule $\phi(\gs^{k} \otimes \gt^{n-k})=-\ge'_{k+1}q(\gt)(\gs) $.
    Note that, if $\sigma \otimes \tau \subset[r] \gD_{X \times X}$ and $q(\tau)(\sigma)\neq 0$, then it follows from the property of $q$ that $\sigma$ and $\tau$ will be $r'$ close to each other where $r'$ depends only on $r$.
    Hence we have $\phi \ccap \gD_{X \times X} \csubset \gD_{\gd(X)}$, i.e.
    $\phi \in \Cx[n](X\otimes X - \gd(X))$.
    Moreover, $\phi$ is a cocycle, since
    \begin{multline*}
        d\phi(\gs^{k} \otimes \gt^{n-k+1})=\phi(\d\gs \otimes \gt + \ge_{k}\gs \otimes \d\gt) = - \ge'_{k}q(\gt)(\d\gs)-\ge_{k}\ge'_{k+1}q(\d\gt)(\gs) \\
        = - (\ge'_{k}+\ge_{k}\ge'_{k+1})dq(\gt)(\gs) =0.
    \end{multline*}

    Then for $x \in X$ we have $\phi(x \otimes c)= -\ge'_{1}q(c)(x)=qp(1)(x)=1(x)=1$.
    Thus by Lemma~\ref{l:opair}(3) $([\phi], [c])$ is an orientation pair.

    Conversely, suppose $X$ has a fundamental class $[X]=[c]$ and an orientation class $U=[\phi]$.
    We claim that the chain maps $\ge'_{n-k-1}\phi/\cdot : \cx[k](X) \to \C[n-k](X)$ and $\ge'_{n}\ge'_{n-k} \cdot \capp c : \C[k](X) \to \cx[n-k](X)$ are controlled chain homotopy equivalences as in the definition of a coarse $PD(n)$ space.

    By Lemma~\ref{l:maps}\eqref{i:capx} $\ge_{n}T^*U \capp ([X]\otimes [X]) = U \capp ([X]\otimes [X]) $, thus by Lemma \ref{l:opair}(2), $\ge_{n}T^*U$ is also a diagonal class for $[X]$ and $T^*U/[X]=\ge_{n}$.

    Note that $ \ge'_{n}\ge'_{n-k}\ge'_{n-k-1}=\ge_{nk}\ge_{n}$.
    Let $a\in \cx[k](X)$.
    Then we have
    \begin{align*}
        \ge_{nk}\ge_{n}(\phi/a)\capp c &= \ge_{n}p_{1*}(\ge_{nk}\phi\capp (c \otimes a))&&\text{by Lemma \ref{l:Spanier 123}\eqref{e:cap} } \\
        &\simeq \ge_{n}p_{1*}(T^{*}\phi\capp (a \otimes c) )&& \text{by Lemma~\ref{l:maps}\eqref{i:capx} } \\
        &=\ge_{n} (T^*\phi/c ) \capp a &&\text{by Lemma \ref{l:Spanier 123}\eqref{e:cap} }\\
        &=a.
    \end{align*}
    Note that the only homotopy we used here is the composition of $p_{1*}$ with the homotopy from Lemma~\ref{l:maps}\eqref{i:capx}, so it has the property $\Supp{D(a)} \csubset \Supp{a}$.

    Similarly, $\ge'_{n}\ge'_{k-1}\ge'_{n-k}=\ge_{nk}$, and for $\phi \in \C[k](X)$ we have
    \begin{align*}
        \ge_{nk} \phi/(\psi\capp c ) &=\ge_{nk} (\phi \cupp p_{2}^{*}(\psi) )/c &&\text{by Lemma \ref{l:Spanier 123}\eqref{e:cupcap} }\\
        &\simeq (p_{1}^{*} ( \psi ) \cupp \phi )/c &&\text{by Lemma~\ref{l:maps}\eqref{i:cupx} } \\
        &= \psi \cupp (\phi/c ) && \text{by Lemma \ref{l:Spanier 123}\eqref{e:scup} }\\
        &=\psi.
    \end{align*}
    Lemma~\ref{l:maps}\eqref{i:cupx} implies that the homotopy above satisfies $\Supp{D(\phi)} \ccap \gD_{X} \csubset \Supp{\phi}$.
\end{proof}
\begin{Corollary}
    Any proper, uniformly acyclic $n$-manifold is a coarse PD($n$) space.
\end{Corollary}
\begin{proof}
    Suppose $X$ is a proper and uniformly acyclic $n$-manifold.
    Since $X$ is an acyclic manifold, it has an orientation class $V \in \H[n](X\times X, X\times X - \gd(X))$ such that $V|_{x \times (X, X - x)}$ is a generator of $\H[n] (x \times (X, X - x))$ for all $x \in X$.
    It also has a fundamental class $[X] \in \hlf[n](X)$ which similarly restricts to generators of $\h[n](X, X - x) $.
    Choosing $V$ and $[X]$ in compatible way we obtain $V( x \times [X] ) =1$ for all $x \in X$.
    Let $\phi'$ and $c'$ be singular (co)cycles representing these classes.
    We now construct their coarse versions.

    By subdividing $c'$ as necessary we can obtain another representative $c$, whose singular simplices have uniformly bounded diameter.
    Since $X$ is proper and $c$ is locally finite, any bounded set intersects $\Supp{c}$ in finitely many simplices, and therefore $c$ is a coarse cycle.

    Using uniform acyclicity of $X\times X$ we can inductively fill coarse simplices by singular chains in a controlled manner.
    Indeed, 0-simplices are naturally identified, and if the faces of a coarse simplex are already filled by singular chains, then their signed sum is a cycle, which by uniform acyclicity bounds a singular chain contained in a controlled neighborhood of the simplex.

    This gives a chain map $B:\cf(X \times X) \to \cs(X \times X)$ with the property $\Supp{B(\gs^{k} ) }\subset[\gr_{k}(\diam(\gs^{k}))] \Supp{\gs^{k} } $, where $\gr_{k}$ is a suitable sequence of control functions.

    Set $\phi (\gs^{n})=\phi'(B(\gs^{n}))$.
    Since $\phi' $ vanishes on singular simplices in $X\times X - \gd(X)$, $\phi(\gs)=0$ if $\gr_{n}(\diam(\gs)) < d(\gs, \gd(X))$.
    Thus $\phi \in \Cx[n](X \times X -\gd(X) )$.

    Let $U:=S^*(\phi)\in \Cx[n](X \otimes X -\gd(X) )$ where $S:\cf(X\otimes X)\to \cf(X\times X)$ is the simplicial cross product map defined in Section~\ref{s:products}.
    By construction, $U(x\otimes c)=\phi( S(x \otimes c ))=\phi(x\times c) =1$ for all $x \in X$.
    Therefore, by Lemma~\ref{l:opair} $(U, [c])$ is an orientation pair, and the claim follows.
\end{proof}
\begin{Remark}
    Essentially the same argument shows that proper, uniformly acyclic ENR homology $n$-manifolds are also coarse PD($n$) spaces.
\end{Remark}

\section{Algebraic duality}\label{s:Algebraic duality}
In this section we show that chains and cochains with $\zz$ coefficients are dual to each other in the sense of taking $\Hom(\cdot, \zz)$.
This requires some algebraic preliminaries, most of which can be found in~\cite{em02}.

Let $R$ be a commutative ring with $1$. 
Consider a countable direct product $R^{\nn}$.
Let $e_{i}$ denote the image in $R^{\nn}$ of $1$ in the $i^{\text{th}}$ factor.
Given an $R$-module $H$ and homomorphism $\phi: R^{\nn} \to H$, the support of $\phi$ is $\Supp{\phi}=\{i \in \nn \mid \phi(e_{i}) \neq 0\}$.
$H$ is \emph{slender} if for every homomorphism $\phi: R^{\nn} \to H$, $\Supp{\phi}$ is finite.
A ring $R$ is slender if it is a slender module over itself.

The following is a classical result, going back to Specker.
For the sake of completeness, we provide a proof.
\begin{Lemma}
    [cf.
    Corollary III.2.4 in \cite{em02}]\label{l:z}
    $\zz$ is slender.
\end{Lemma}
\begin{proof}
    Let $\phi:\zz^\nn \to \zz$ be a homomorphism.
    We will produce an element $v\in \zz^\nn$ of the form $\sum 2^{n_{k}} e_{k}$ with fast growing powers of $2$ such that $\phi(e_{k+1})=0$ for $k\geq\abs{\phi(v)}$, and hence $\phi$ has finite support.
    (Here and below $\abs{ \phi (\cdot)}$ denotes the absolute value of $\phi (\cdot)$.)

    For each $k$, let $v = h_k + t_k$ be the splitting of $v$ into the head of length $k-1$ and the infinite tail.
    At each step $k$ we already have $h_k$ (starting with $h_{1}=0$), and we choose $n_k > n_{k-1}$ so that $2^{n_{k}} > k + \abs{ \phi ( h_{k} ) }$.

    Now let $k \geq \abs{ \phi (v) }$ be arbitrary.
    Since $\phi (v) = \phi (h_k) + \phi (t_k)$, we have
    \[
    \abs{ \phi (t_k) } \le \abs{\phi (v) } + \abs{ \phi (h_k) } \le k + \abs{ \phi (h_k) } < 2^{n_{k}}.
    \]
    But $\phi (t_k)$ is a multiple of $2^{n_k}$, since $t_k$ is.
    So $\phi ( t_k )=0$ and therefore $2^{n_{k+1}} \phi (e_{k+1}) = \phi (t_{k+1}) - \phi (t_{k})=0$, proving the claim.
\end{proof}
\begin{Theorem}
    [cf.
    Corollary III.3.4 in \cite{em02}]\label{t:slender}
    If $H$ is slender and $I$ is not an $\go$-measurable cardinal, then for any $\phi: R^{I} \to H$, if $\phi_{|\bigoplus_{I}R}=0$, then $\phi=0$.
\end{Theorem}

We recall that a cardinal is \emph{$\go$-measurable} if it supports a non-principal ultrafilter which is closed under countable intersections.
The existence of $\go$-measurable cardinals cannot be proved in ZFC, and it is unknown if the existence is consistent with ZFC.
Moreover, if they exist, the first $\go$-measurable cardinal $\gk$ must be strongly inaccessible, i.e.
if $\gl < \gk$, then $2^{\gl} < \gk$ and $\gk$ cannot be obtained as $\gl$-union of $\gl$ cardinals.

Let $\cS$ be a family of subsets of $I$ which is closed under taking subsets and finite unions.
These conditions ensure that the set of $\cS$-supported functions $M_{\cS}=\{ a : I \to R \mid \Supp{a} \in \cS \}$ is a submodule of $R^{I}$, so we will call such $\cS$ a \emph{supportive} family.
The dual of $\cS$, given by $\cS^{*}=\{ T \subset I \mid \forall S \in \cS \quad \#(S \cap T)<\infty \}$ is also supportive.
Note that there is a natural bilinear pairing $\langle \, , \rangle: M_{\cS}\times M_{\cS^{*}} \to R$ given by $\langle a, b\rangle = \sum_{i\in I}a(i)b(i)$.
\begin{Lemma}\label{l:slender}
    Suppose $R$ is slender, the cardinality $I$ is less than the first $\go$-measurable cardinal, and $\cS$ is a supportive family which covers $I$.
    Then the evaluation map $e: M_{\cS^{*}} \to \Hom(M_{\cS} , R)$ given by $e(b)(a)=\langle a, b \rangle$ is an isomorphism.
\end{Lemma}
\begin{proof}
    The injectivity of $e$ follows easily from the hypothesis that $\cS$ covers $I$.

    To prove surjectivity, let $f \in \Hom(M_{\cS} , R)$.
    We claim that $\Supp{f} \in \cS^{*}$.
    Indeed, if the intersection $\Supp{f} \cap S$ is infinite for some $S \in \cS$, then it contains an infinite countable subset $C$, and thus $R^{C} \subset M_{\cS} $ and $\Supp{f_{|R^{C}}}=C$, contradicting slenderness of $R$.

    We also claim that $\Supp{f}=\emptyset$ implies $f=0$.
    If $f(a)\neq 0$ for some $a \in M_{\cS}$, then since $M_{\cS}$ contains $R^{\Supp{a} }$ we get a contradiction to Theorem~\ref{t:slender} for $f|_{R^{\Supp{a} } }$.

    The first claim shows that the characteristic function $1_{\Supp{f}} \in M_{\cS^{*}}$, and the second claim implies that $f-e(1_{\Supp{f}} )= 0$.
\end{proof}

Let $X$ be a metric space with two subspaces $A$ and $B$.
The families
\begin{align*}
    \cH &=\{S \subset X \mid S \csubset B \text{ and } \forall r \quad \#( S \cap N_{r} (A) ) < \infty \}, \\
    \cC &=\{T \subset X \mid T \ccap B \csubset A \}
\end{align*}
are supportive and cover $X$.
\begin{Lemma}\label{l:dual0}
    The families $\cH$ and $\cC$ are dual to each other: $\cH^{*}=\cC$ and $\cC^{*}=\cH$.
\end{Lemma}
\begin{proof}
    If $Y$ belongs to both families, then $Y \csubset B$, by the first condition of $\cH$.
    Then $Y \ceq Y \ccap B$, so $Y \csubset A$, by definition of $\cC$.
    So $Y$ has to be finite by the second condition of $\cH$.
    This implies that if $S \in \cH$ and $T \in \cC$, then $S \cap T$ is finite, which means $\cH^{*} \supset \cC$ and $\cC^{*} \supset \cH$.

    Now let $Y\in \cC^{*} $.
    Since $\cC$ contains all subsets which are coarsely disjoint from $B$, $Y$ cannot contain an infinite subset which is coarsely disjoint from $B$.
    This implies that $Y \csubset B$.
    Since $\cC$ also contains all metric neighborhoods of $A$, $Y$ must also satisfy the second condition of $\cH$, so $Y \in \cH$.
    Thus $\cC^{*} \subset \cH$.

    Finally, if $Y \notin \cC$, then $Y$ contains an infinite locally finite subset which is coarsely disjoint form $A$ and coarsely contained in $B$.
    This subset is in $\cH$, so $Y \notin \cH^{*}$.
    Thus $\cH^{*} \subset \cC$.
\end{proof}

From now on we will always assume that the cardinality of our space $X$ is less than the first $\go$-measurable cardinal.
\begin{Lemma}\label{l:dual1}
    The evaluation map $e$ of cochains on chains induces isomorphisms of (co)chain complexes:
    \begin{align*}
        \cx(B-A \subset X) &= \Hom(\Cx(B-A \subset X), \zz) , \\
        \Cx(B-A \subset X) &= \Hom(\cx(B-A \subset X), \zz). \\
    \end{align*}
\end{Lemma}
\begin{proof}
    Suppose $R=\zz$, $X=X^{n+1}$, and $B=\gD_B$ and $A=\gD_A$.
    Then the corresponding $\cH$ is the collection of supports of chains in $\cx(B-A\subset X)$.
    It follows that $M_\cH=\cx(B-A\subset X)$.
    For a similar reason, $M_\cC=\Cx(B-A\subset X))$.
    So, we need to show that $M_{\cH}=\Hom(M_{\cC}, \zz)$ and $M_{\cC}=\Hom(M_{\cH}, \zz)$.

    Lemma~\ref{l:slender} gives $M_{\cC^*}=\Hom(M_{\cC}, \zz)$ and Lemma~\ref{l:dual0} gives $\cC^*=\cH$.
    Combining these, we get $M_{\cH}=\Hom(M_{\cC}, \zz)$.
    Similarly, we can show that $M_{\cC}=\Hom(M_{\cH}, \zz)$.
    Finally, it is clear that the evaluation map commutes with the differentials.
\end{proof}

\end{document}